\documentclass[a4paper]{amsart}

\usepackage{amsthm,amsfonts,amsmath, amssymb}

\usepackage{todonotes}

\usepackage{isomath}
\usepackage{lmodern} 
\usepackage[italic]{mathastext}

\usepackage{setspace}
\usepackage{graphicx,float,enumerate,subfigure,tikz}
\usepackage{algpseudocode}
\usepackage{verbatim}
\usepackage{url}
\usepackage{epstopdf}
\usepackage[capitalise]{cleveref}

\newtheorem{theorem}{Theorem}
\newtheorem{corollary}[theorem]{Corollary}
\newtheorem{lemma}[theorem]{Lemma}
\newtheorem{proposition}[theorem]{Proposition}

\newenvironment{claimproof}{\noindent\textit{Proof.}}{\hfill$\square$}

\theoremstyle{definition}
\newtheorem{definition}[theorem]{Definition}

\theoremstyle{remark} 
\newtheorem{remark}[theorem]{Remark}

\crefname{remark}{Remark}{Remarks}

\theoremstyle{definition}
\newtheorem{claim}{Claim}
\usepackage{etoolbox}
\AtEndEnvironment{proof}{\setcounter{claim}{0}}

\theoremstyle{remark}
\newtheorem{case}{Case}
\usepackage{etoolbox}
\AtEndEnvironment{proof}{\setcounter{case}{0}}
\AtEndEnvironment{claimproof}{\setcounter{case}{0}}

\usepackage{mathtools}

\DeclarePairedDelimiter\floor{\lfloor}{\rfloor}
\DeclareMathOperator{\dist}{dist}

\DeclareMathOperator{\conv}{conv}

\DeclareMathOperator{\st}{star}
\DeclareMathOperator{\a-st}{astar}
\DeclareMathOperator{\lk}{link}

\newcommand{\R}{\mathbb{R}}

\newcommand{\A}{\mathcal{A}}

\newcommand{\C}{\mathcal{C}}
\newcommand{\B}{\mathcal{B}}

\newcommand{\St}{\mathcal{S}}
\newcommand{\Lk}{\mathcal{L}}

\newcommand{\dmark}{\textsuperscript{\dag}}

\crefname{rmk}{Remark}{Remarks}
\crefname{problem}{Problem}{Problems}
\date{\today}
\title{The linkedness of cubical polytopes: beyond the cube}

\usepackage{amsaddr}

\author{Hoa T. Bui}
\address{Federation University Australia\\
Faculty of Science and Engineering, Curtin University, Australia}
\email{\texttt{hoa.bui@curtin.edu.au}}

\author{Guillermo Pineda-Villavicencio \& Julien Ugon}
\address{Federation University Australia\\School of Information Technology, Deakin University, Australia}
\email{\texttt{julien.ugon@deakin.edu.au}} 
\email{\texttt{work@guillermo.com.au}}

\thanks{Hoa T. Bui is supported by an Australian Government Research Training Program (RTP) Stipend and RTP Fee-Offset Scholarship through Federation University Australia. Julien Ugon's research was partially supported by ARC discovery project DP180100602.}

\begin{document}
\begin{abstract} A cubical polytope is a polytope with all its facets being combinatorially equivalent to cubes. The paper is concerned with the linkedness of the graphs of cubical polytopes.

 A graph with at least $2k$ vertices is \textit{$k$-linked} if, for every set of $k$ disjoint pairs of vertices, there are $k$ vertex-disjoint paths joining the vertices in the pairs. We say that a polytope is \textit{$k$-linked} if its graph is $k$-linked.
 In a previous paper \cite{BuiPinUgo20a} we proved that every cubical $d$-polytope is $\floor{d/2}$-linked. Here we strengthen this result by establishing the $\floor{(d+1)/2}$-linkedness of cubical $d$-polytopes, for every $d\ne 3$.

A graph $G$ is {\it strongly $k$-linked} if it has at least $2k+1$ vertices and, for  every vertex $v$ of $G$, the subgraph $G-v$ is $k$-linked.
 We say that a polytope is (strongly) \textit{$k$-linked} if its graph is (strongly) $k$-linked. In this paper, we also prove that every cubical $d$-polytope is strongly $\floor{d/2}$-linked, for every $d\ne 3$.

These results are best possible for this class of polytopes.
\end{abstract}
\maketitle

\section{Introduction}

The {\it graph} $G(P)$ of a polytope $P$ is the undirected graph formed by the vertices and edges of the polytope. This paper studies the linkedness of {\it cubical $d$-polytopes}, $d$-dimensional polytopes with all their facets being cubes. A \textit{$d$-dimensional cube} is the convex hull in $\R^{d}$ of the $2^{d}$ vectors $(\pm 1,\ldots,\pm 1)$. By a cube we mean any polytope whose face lattice is isomorphic to the face lattice of a cube.   
 
Denote by $V(X)$ the vertex set of a graph or a polytope $X$. Given sets $A,B$ of vertices in a graph, a path from $A$ to $B$, called an {\it $A-B$ path}, is a (vertex-edge) path $L:=u_{0}\ldots u_{n}$ in the graph such that $V(L)\cap A=\{u_{0}\}$  and $V(L)\cap B=\{u_{n}\}$. We write $a-B$ path instead of $\{a\}-B$ path, and likewise, write $A-b$ path instead of $A-\{b\}$ path. 

Let $G$ be a graph and $X$ a subset of $2k$ distinct vertices of $G$. The elements of $X$ are called {\it terminals}. Let $Y:=\{\{s_{1},t_{1}\}, \ldots,\{s_{k},t_{k}\}\}$ be an arbitrary labelling and (unordered) pairing of all the vertices in $X$. We say that $Y$ is {\it linked} in $G$ if we can find disjoint $s_{i}-t_{i}$ paths for all $i\in [1,k]$, where $[1,k]$ denotes the interval $1,\ldots,k$. The set $X$ is {\it linked} in $G$ if every such pairing of its vertices is linked in $G$. Throughout this paper, by a set of disjoint paths, we mean a set of vertex-disjoint paths. If $G$ has at least $2k$ vertices and every set of exactly $2k$ vertices is linked in $G$, we say that $G$ is {\it $k$-linked}. If the graph of a polytope is $k$-linked, we say that the polytope is also {\it $k$-linked}.

Linkedness is a stronger property than connectivity: let $G$ be a graph with at least $2k$ vertices, and  let $S:=\{s_{1},\ldots,s_{k}\}$ and $T:=\{t_{1},\ldots,t_{k}\}$ be two disjoint $k$-element sets of vertices in $G$. It follows from Menger's theorem that, if $G$ is $k$-connected then the sets $S$ and $T$ can be joined \textbf{setwise}  by disjoint paths (namely, by $k$ disjoint $S-T$ paths). By contrast, if $G$ is $k$-linked then the sets can be joined \textbf{pointwise} by disjoint paths.  

A closely related problem to linkedness is the classical \textit{disjoint paths problem} \cite{RobSey-XIII}: given a graph $G$ and a set  $Y:=\{\{s_{1},t_{1}\}, \ldots,\{s_{k},t_{k}\}\}$ of $k$ pairs of terminals in $G$, decide whether or not $Y$ is linked in $G$. A natural optimisation version of this problem is to find the largest subset of the pairs so that there exist disjoint paths connecting the selected pairs.

There is a linear function $f(k)$ such that every $f(k)$-connected graph is $k$-linked, which follows from works of Bollob\'as and Thomason \cite{BolTho96}; Kawarabayashi, Kostochka, and Yu \cite{KawKosYu06}; and Thomas and Wollan \cite{ThoWol05}. In the case of polytopes, Larman and Mani \cite[Thm.~2]{LarMan70} proved that every $d$-polytope  is $\floor{(d+1)/3}$-linked, a result that was slightly improved to $\floor{(d+2)/3}$  in \cite[Thm.~2.2]{WerWot11}. Gallivan \cite{Gal85} proved that not every polytope is $\lfloor d/2\rfloor$-linked. 
In view of this negative result, researchers have focused efforts on finding families of $d$-polytopes that are $\floor{d/2}$-linked.  
In his PhD thesis \cite[Question~5.4.12]{Ron09}, Wotzlaw asked whether every cubical $d$-polytope is $\floor{d/2}$-linked. In \cite{BuiPinUgo20a} we answer his question in the affirmative  by establishing the following theorem.
\begin{theorem} 
\label{prop:weak-linkedness-cubical} For every $d\ge 1$, a  cubical $d$-polytope is $\floor{d/2}$-linked.
\end{theorem}

The paper \cite{BuiPinUgo20a} also established the linkedness of the $d$-cube.

\begin{theorem}[{Linkedness of the cube}]\label{thm:cube-linkedness}  For every $d\ne 3$, a $d$-cube is $\floor{(d+1)/2}$-linked.
\end{theorem}

In this paper, we extend these two results as follows:
\begin{theorem}[Linkedness of cubical polytopes]\label{thm:cubical} For every $d\ne 3$, a cubical $d$-polytope is $\floor{(d+1)/2}$-linked.
\end{theorem}

Our methodology relies on results on the connectivity of strongly connected subcomplexes of cubical polytopes, whose proof ideas were first developed in \cite{ThiPinUgo18v3}, and a number of new insights into the structure of $d$-cube exposed in \cite{BuiPinUgo20a}. One obstacle that forces some tedious analysis is the fact that the 3-cube is not 2-linked.

Let  $X$ be a set of vertices in a graph $G$. Denote by $G[X]$ the subgraph of $G$ induced by $X$, the subgraph of $G$ that contains all the edges of $G$ with vertices in $X$. Write $G-X$ for $G[V(G)\setminus X]$. If $X=\{v\}$, then we write $G-v$ instead of $G-\{v\}$.

In our paper \cite{BuiPinUgo20a}, we introduce the notion of strong linkedness. We say that a graph $G$ with at least $2k+1$ vertices is {\it strongly $k$-linked} if for  every vertex $v$ of $G$, the subgraph $G-v$ is $k$-linked.
 A polytope is  \textit{strongly $k$-linked} if its graph is strongly $k$-linked. We proved the strong-linkedness of the cube as follows:
 \begin{theorem}[Strong linkedness of the cube {\cite[Thm. 25]{BuiPinUgo20a}}]\label{thm:cube-strong-linkedness}  For every $d\ge 1$, a $d$-cube is strongly $\floor{d/2}$-linked. 
\end{theorem} 

 In this paper, we extend this result to cubical polytopes:

\begin{theorem}[Strong linkedness of cubical polytopes]\label{thm:cubical-strong-linkedness} For every $d\ne 3$, a cubical $d$-polytope is strongly $\floor{d/2}$-linked.\end{theorem} 

Unless otherwise stated, the graph theoretical notation and terminology follow from \cite{Die05} and the polytope theoretical notation and terminology from \cite{Zie95}. Moreover, when referring to graph-theoretical properties of a polytope such as minimum degree, linkedness and connectivity, we mean properties of its graph.

\section{Connectivity of cubical polytopes}
\label{sec:cubical-connectivity}

The aim of this section is to present a couple of results related to the connectivity of strongly connected complexes in cubical polytopes.   
A pure polytopal complex $\C$ is {\it strongly connected} if every pair of facets $F$ and $F'$ is connected by a path $F_{1}\ldots F_{n}$ of facets in $\C$ such that $F_{i}\cap F_{i+1}$ is a ridge of $\C$ for each $i\in [1,n-1]$, $F_{1}=F$ and $F_{n}=F'$; we say that such a path is a {\it $(d-1,d-2)$-path} or a {\it facet-ridge path} if the dimensions of the faces can be deduced from the context. Two basic examples of strongly connected complexes are given by the complex of all faces of a polytope $P$, called the {\it complex} of $P$ and denoted by $\C(P)$, and the complex of all proper faces of $P$, called the {\it boundary complex} of $P$ and denoted by $\B(P)$. For the definitions of polytopal complexes and pure polytopal complexes, refer to~\cite[Section 5.1]{Zie95}.

Given  a polytopal complex $\C$ with vertex set $V$ and a subset $X$ of $V$,  the subcomplex of $\C$ formed by all the faces of $\C$ containing only vertices from $X$ is said to be {\it induced by $X$} and is denoted by $\C[X]$.  Removing from $\C$ all the vertices in a subset $X\subset V(\C)$  results in the subcomplex $\C[V(\C)\setminus X]$, which we write as $\C-X$. If $X=\{x\}$ we write $\C-x$ rather than $\C-\{x\}$. We say that a subcomplex $\C'$ of a complex $\C$ is a {\it spanning} subcomplex of $\C$ if $V(\C')=V(\C)$. The {\it graph} of a complex is the undirected graph formed by the vertices and edges of the complex; as in the case of polytopes, we denote the graph of a complex $\C$ by $G(\C)$. 

For a polytopal complex $\C$, the {\it star} of a face $F$ of $\C$, denoted $\st(F,\C)$, is the subcomplex of $\C$ formed by all the faces containing $F$, and their faces; the {\it antistar} of a face $F$ of $\C$, denoted $\a-st(F,\C)$, is the subcomplex of $\C$ formed by all the faces disjoint from $F$; and the {\it link} of a face $F$, denoted $\lk(F,\C)$, is the subcomplex of $\C$ formed by all the faces of $\st(F,\C)$ that are disjoint from $F$. That is, $\a-st(F,\C)=\C-V(F)$ and $\lk(F,\C)=\st(F,\C)-V(F)$. Unless otherwise stated, when defining stars, antistars and links in a polytope, we always assume that the underlying complex is the boundary complex of the polytope.   

The first results are from  \cite{ThiPinUgo18v3}.

\begin{lemma}[{\cite[Lem.~8]{ThiPinUgo18v3}}]
\label{lem:cube-face-complex} Let $F$ be a proper face in the $d$-cube $Q_{d}$. Then the antistar of $F$ is a strongly connected $(d-1)$-complex. 
\end{lemma}

\begin{proposition}[{\cite[Prop.~13]{ThiPinUgo18v3}}]\label{prop:star-minus-facet}  Let $F$ be a facet in the star $\St$  of a vertex in a cubical $d$-polytope. Then the antistar of $F$ in $\St$ is a strongly connected $(d-2)$-subcomplex of $\St$. 
\end{proposition}

 Let $v$ be a vertex in a $d$-cube $Q_{d}$ and let $v^{o}$ denote the vertex at distance $d$ from $v$, called the vertex {\it opposite} to $v$ in $Q_{d}$; by distance in a cube, we mean the graph-theoretical distance in the cube.   
In the $d$-cube $Q_{d}$, the facet disjoint from a facet $F$ is denoted by $F^o$, and we say that $F$ and $F^{o}$ are a pair of {\it opposite} facets. 

We proceed with a simple but useful remark.
 
\begin{remark}
\label{rmk:opposite-vertex-1}
Let $P$ be a cubical $d$-polytope. Let  $v$  be a vertex  of $P$ and let $F$ be a face of $P$ containing $v$, which is a cube. In addition, let $v^{o}$ be the vertex of $F$ opposite to $v$ in $F$. The smallest face in the polytope containing both $v$ and $v^{o}$ is precisely $F$.
\end{remark}


The proof idea in \cref{prop:star-minus-facet} can be pushed a bit further to obtain a rather technical result that we prove next. Two vertex-edge paths are {\it independent} if they share no inner vertex.

\begin{lemma} Let $P$ be a cubical $d$-polytope with $d\ge 4$. Let $s_{1}$ be any vertex in $P$ and let $\St_{1}$ be the star of $s_{1}$ in the boundary complex of $P$.  Let $s_{2}$ be any vertex in $\St_{1}$, other than $s_{1}$. Define the following sets:
\begin{itemize}
\item $F_{1}$ in $\St_{1}$, a facet containing $s_{1}$ but not $s_{2}$;
\item  $F_{12}$ in $\St_{1}$, a facet containing $s_{1}$ and $s_{2}$;
\item $\St_{12}$, the star of $s_{2}$ in $\St_{1}$ (that is, the subcomplex of $\St_{1}$ formed by the facets of $P$ in $\St_{1}$ containing $s_{2}$);
\item $\mathcal A_{{1}}$, the antistar of $F_{1}$ in $\St_{1}$; and
\item $\mathcal A_{12}$, the subcomplex of $\St_{12}$ induced by $V(\St_{12})\setminus (V(F_{1})\cup V(F_{12}))$.
\end{itemize}
Then the following assertions hold.
\begin{enumerate}[(i)]
\item The complex $\St_{12}$ is a  strongly connected $(d-1)$-subcomplex of $\St_{1}$.
\item If there are more than two facets in $\St_{12}$, then, between any two facets of $\St_{12}$ that are different from $F_{12}$, there exists a $(d-1,d-2)$-path in $\St_{12}$ that does not contain the facet $F_{12}$.

\item If $\St_{12}$ contains more than one facet, then the subcomplex $\mathcal A_{{12}}$ of $\St_{12}$ contains a spanning strongly connected $(d-3)$-subcomplex.
\end{enumerate}
\label{lem:technical}
\end{lemma} 

\begin{proof} Let us prove (i). Let $\psi$ define the natural anti-isomorphism from the face lattice of $P$ to the face lattice of its dual $P^{*}$. The facets in $\St_{1}$ correspond to the vertices in the facet $\psi(s_{1})$ in $P^{*}$ corresponding to $s_{1}$; likewise for the facets in $\st(s_{2},\B(P))$ and the vertices in $\psi(s_{2})$. The facets in $\St_{12}$ correspond to the vertices in the nonempty face $\psi(s_{1})\cap \psi(s_{2})$ of $P^{*}$. The existence of a facet-ridge path in $\St_{12}$ between any two facets $J_{1}$ and $J_{2}$ of $\St_{12}$ amounts to the existence of a vertex-edge path in $\psi(s_{1})\cap \psi(s_{2})$ between $\psi(J_{1})$ and $\psi(J_{2})$. That $\St_{12}$ is a strongly connected $(d-1)$-complex now follows from the connectivity of the graph of $\psi(s_{1})\cap \psi(s_{2})$ (Balinski's theorem), as desired.

We proceed with the proof of (ii). Let $J_{1}$ and $J_{2}$ be two facets of $\St_{12}$, other than $F_{12}$. If there are more than two facets in $\St_{12}$, then the face $\psi(s_{1})\cap \psi(s_{2})$ is at least bidimensional. As a result, the graph of $\psi(s_{1})\cap \psi(s_{2})$ is at least 2-connected by Balinski's theorem. By Menger's theorem, there are at least two independent vertex-edge paths in $\psi(s_{1})\cap \psi(s_{2})$ between $\psi(J_{1})$ and $\psi(J_{2})$. Pick one such path $L^{*}$  that avoids the vertex $\psi(F_{12})$ of $\psi(s_{1})\cap \psi(s_{2})$. Dualising this path $L^{*}$ gives a $(d-1,d-2)$-path between $J_{1}$ and $J_{2}$ in  $\St_{12}$ that does not contain the facet $F_{12}$.

We finally prove (iii). Assume that $\St_{12}$ contains more than one facet. We need some additional notation.
\begin{itemize}
\item  Let $F$ be a facet in $\St_{12}$ other than $F_{12}$; it exists by our assumption on $\St_{12}$.
\item For a facet $J$ in $\St_{12}$, let $\A_{1}^{J}$ denote the subcomplex  $J-V(F_{1})$; that is, $\A_{1}^{J}$ is the antistar of $J\cap F_{1}$ in $J$.  
\item For a facet $J$ in $\St_{12}$ other than $F_{12}$, let $\A_{12}^{J}$ denote the subcomplex $J-(V(F_{1})\cup V(F_{12}))$, the subcomplex of $J$ induced by $V(J)\setminus (V(F_{1})\cup V(F_{12}))$.  
\end{itemize}
We require the following claim. 
 
\begin{claim} \label{cl:tech-lem-claim2} $\A_{12}^{F}$ contains a spanning strongly connected $(d-3)$-subcomplex $\C^{F}$. 
\end{claim}	
\begin{claimproof}  We first show that $\A_{12}^{F}\ne \emptyset$. Denoting by $s_{1}^{o}$ the vertex in $F$ opposite to $s_{1}$, we have that $s_{1}^{o}$ is not in $F_{1}$ or in $F_{12}$ by \cref{rmk:opposite-vertex-1}. So $s_{1}^{o}$  is in $\A_{12}^{F}$.  
 
Notice that $s_{1}\not \in \A_{1}^{F}$. From \cref{lem:cube-face-complex} it follows that $\A_{1}^{F}$ is a strongly connected $(d-2)$-subcomplex of $F$. Write \[\A_{1}^{F}=\C(R_{1})\cup\cdots \cup \C(R_{m}),\] where $R_{i}$ is a $(d-2)$-face of $F$ for each $i\in [1,m]$. Every $(d-2)$-face in $F$ contains either $s_{1}$ or $s_{1}^{o}$, and since we have $s_{ 1}\not\in R_{i}$ for every $R_i\in \A_{1}^F$, it follows that $s_{1}^o\in R_i$. Consequently no ridge $R_{i}$ is contained in $F_{12}$. 

Let $$\C_{i}:=\B(R_{i})-V(F_{12}).$$ As $R_{i}\not\subset F_{12}$, we have $\dim R_{i}\cap F_{12}\le d-3$. Furthermore, since $s_1^o\in \C_{i}$, $\C_{i}$ is nonempty.  If $R_{i}\cap F_{12}\ne \emptyset$, then  $\C_{i}$ is  the antistar  of $R_{i}\cap F_{12}$ in $R_{i}$, a spanning strongly connected $(d-3)$-subcomplex of $R_{i}$ by \cref{lem:cube-face-complex}. If $R_{i}\cap F_{12}= \emptyset$, then $\C_{i}$ is the boundary complex of $R_{i}$, again a spanning strongly connected $(d-3)$-subcomplex of $R_{i}$.

Let \[\C^{F}:=\bigcup \C_{i}.\] Then the complex $\C^{F}$ is a spanning $(d-3)$-subcomplex of $\A_{12}^{F}$; we show it is strongly connected. 

Take any two $(d-3)$-faces $W$ and $W'$ in $\C^{F}$. We find a $(d-3,d-4)$-path $L$ in $\C^{F}$ between $W$ and $W'$. There exist ridges $R$ and $R'$ in $\A_{1}^{F}$ with $W\subset R$ and $W'\subset R'$.  Since $\A_{1}^{F}$ is a strongly connected $(d-2)$-complex, there is a $(d-2,d-3)$-path $R_{i_1}\ldots R_{i_{p}}$ in $\A_{1}^{F}$ between $R_{i_1}=R$ and $R_{i_{p}}=R'$, with $R_{i_{j}}\in \A_{1}^{F}$ for each $j\in [1,p]$. We will show by induction on the length $p$ of the $(d-2,d-3)$-path $R_{i_1}\ldots R_{i_{p}}$ that there is a $(d-3,d-4)$-path in $\C^F$ between $W$ and $W'$.

If $p=1$, then $R_{i_{1}} = R_{i_{p}} = R = R'$. The existence of the path follows from the strong connectivity of $\C_{i_{1}}$.

Suppose that the claim is true when the length of the path is $p-1$.  We  already established that $s_{1}^{o}\in R_{i_{j}}$ for every $j\in [1,p]$ and that $s^{o}_{1}\not \in F_{12}$. Consequently, we get that $R_{i_{p-1}}\cap R_{i_{p}}\not \subset F_{12}$, and therefore, $R_{i_{p-1}}\cap R_{i_{p}}\cap F_{12}$ is a proper face of $R_{i_{p-1}}\cap R_{i_{p}}$. Hence  the subcomplex $\B_{i_{p-1}}:=\B(R_{i_{p-1}}\cap R_{i_{p}})-V(F_{12})$ of $\B(R_{i_{p-1}}\cap R_{i_{p}})$ is a nonempty, strongly connected $(d-4)$-complex by \cref{lem:cube-face-complex}; in particular, it contains a $(d-4)$-face $U_{i_{p}}$.  Furthermore, $\B_{i_{p-1}}\subset \C_{i_{p-1}}\cap \C_{i_{p}}$.

Let $W_{i_{p-1}}$ and $W_{i_{p}}$ be $(d-3)$-faces in $\C_{i_{p-1}}$ and $\C_{i_{p}}$ containing $U_{i_{p}}$ respectively. By the induction hypothesis, the existence of the $(d-2,d-3)$-path $R_{i_1}\ldots R_{i_{p-1}}$ implies the existence of a $(d-3,d-4)$-path $L_{p-1}$ in $\C^F$ from $W$ to $W_{i_{p-1}}$. The strong connectivity of $\C_{i_{p}}$ gives the existence of a path $L_{p}$ from $W_{i_{p}}$ to $W'$. Finally, the desired $(d-3,d-4)$-path $L$ is the concatenation of these two paths: $L=L_{p-1}W_{i_{p-1}}U_{i_p}W_{i_p}L_p$. The existence of the path $L$ between $W$ and $W'$ completes the proof of \cref{cl:tech-lem-claim2}. 
\end{claimproof}

We are now ready to complete the proof of (iii). The proof goes along the lines of the proof of \cref{cl:tech-lem-claim2}. We let 
\[\St_{12}=\bigcup_{i=1}^{m} \C(J_{i}),\] where the facets $J_{1},\ldots,J_{m}$ are all the facets in $P$ containing $s_{1}$ and $s_{2}$. 

For every $i\in [1,m]$ we let  $\C^{J_{i}}$ be the spanning strongly connected $(d-3)$-subcomplex in $\A_{12}^{J_{i}}$ given by ~\cref{cl:tech-lem-claim2}. And we let \[\C:=\bigcup \C^{J_{i}}.\] Then $\C$ is a spanning $(d-3)$-subcomplex of $\A_{12}$; we show it is strongly connected. 

If there are exactly two facets in $\St_{12}$, namely $F_{12}$ and some other facet $F$, then the complex $\A_{12}$ coincides with the complex $\A_{12}^{F}$. The strong $(d-3)$-connectivity of $\C$ is then settled by \cref{cl:tech-lem-claim2}. Hence assume that there are more than two facets in  $\St_{12}$; this implies that the smallest face containing $s_{1}$ and $s_{2}$ in $\St_{12}$ is at most $(d-3)$-dimensional. 

Take any two $(d-3)$-faces $W$ and $W'$ in $\C$.  Let $J\ne F_{12}$ and $J'\ne F_{12}$ be facets of $\St_{12}$ such that $W\subset J$ and $W'\subset J'$. By (ii), we can find a $(d-1,d-2)$-path $J_{i_{1}}\ldots J_{i_{q}}$ in $\St_{12}$ between $J_{i_{1}}=J$ and $J_{i_{q}}=J'$ such that $J_{i_{j}}\ne F_{12}$ for any $j\in[1,q]$. We will show that a $(d-3,d-4)$-path $L$ exists between $W$ and $W'$ in $\C$, using an induction on the length $q$ of the path $J_{i_{1}}\ldots J_{i_{q}}$.

If $q=1$, then $W$ and $W'$ belong to the same facet $F$ in $\St_{12}$, which is different from $F_{12}$. In this case, $W$ and $W'$  are both in $\A_{12}^{F}$, and consequently, \cref{cl:tech-lem-claim2} gives the desired $(d-3,d-4)$-path between $W$ and $W'$ in $\A_{12}^{F}\subseteq \C$.

Suppose that the induction hypothesis holds when the length of the path is $q-1$. 
First, we show that there exists a $(d-4)$-face $U_{q}$ in $C^{J_{i_{q-1}}}\cap C^{J_{i_{q}}}$. As $J_{i_{q-1}},J_{i_{q}}\ne F_{12}$, we obtain that $\B(J_{i_{q-1}}\cap J_{i_{q}})-V(F_{12})$ is a nonempty, strongly connected $(d-3)$-subcomplex (\cref{lem:cube-face-complex}); in particular, it contains a $(d-3)$-face $K_{q}$. The complex  $\B(K_{q})-V(F_{1})$ is nonempty because  $s_{1}\in F_1$ and $s_{1}\notin K_{q}$ (since $K_{q}$ does not contain any vertex from $F_{12}$). Therefore $\B(K_{q})-V(F_{1})$ is a  strongly connected $(d-4)$-subcomplex by \cref{lem:cube-face-complex}. In particular, $\B(K_{q})-V(F_{1})$ contains a $(d-4)$-face $U_{q}$.


Pick $(d-3)$-faces $W_{q-1}\in \C^{J_{i_{q-1}}}$ and $W_{q}\in \C^{J_{i_{q}}}$ such that both contain the $(d-4)$ face $U_{q}$. The induction hypothesis tells us that there exists a $(d-3,d-4)$-path $L_{q-1}$ from $W$ to $W_{q-1}$ in $\C$. And the strong $(d-3)$-connectivity of $\C^{J_{i_{q}}}$ ensures that there exists a $(d-3,d-4)$-path $L_q$ from $W_{q}$ to $W'$. By concatenating these two paths, we can obtain the path $L=WL_{q-1}W_{q-1}U_{q}W_{q}L_{q}W'$. This completes the proof of the lemma.\end{proof}

\section{Linkedness of cubical polytopes}
\label{sec:cubical-linkedness}	 

The aim of this section is to prove that, for every $d\ne 3$, a cubical $d$-polytope is $\floor{(d+1)/2}$-linked (\cref{thm:cubical}). It suffices to prove \cref{thm:cubical} for odd $d\ge 5$; since $\floor{d/2}=\floor{(d+1)/2}$ for even $d$, \cref{prop:weak-linkedness-cubical} trivially establishes \cref{thm:cubical} in this case.

The proof of \cref{thm:cubical} heavily relies on \cref{lem:star-cubical}. To state the lemma we require the following definition.

 \begin{definition}[Configuration $d$F]\label{def:Conf-dF} Let $d\ge 3$ be odd and let $X$ be a set of at least $d+1$ terminals in a cubical $d$-polytope $P$. In addition, let $Y$ be a  labelling and pairing of the vertices in $X$. A terminal of $X$, say $s_{1}$, is in {\it Configuration $d$F} if the following conditions are satisfied:
 \begin{enumerate}
 \item[(i)]  at least $d+1$ vertices of $X$ appear in a facet $F$ of $P$;
 \item[(ii)]  the terminals in the pair $\{s_{1}, t_{1}\}\in Y$ are at distance $d-1$ in $F$ (that is, $\dist_{F}(s_{1},t_{1})=d-1$); and
 \item[(iii)] the neighbours of $t_{1}$ in $F$ are all vertices of $X$. 
 \end{enumerate}
\end{definition}
 
Figure~\ref{fig:conf-Df} illustrates examples of Configuration $d$F.
\begin{figure}  
\includegraphics{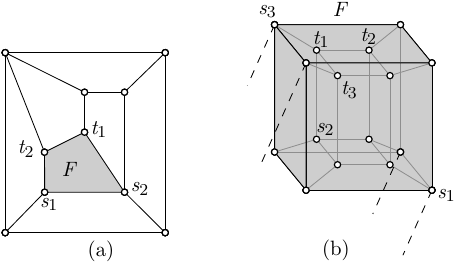}
\caption{Examples of Configuration $dF$. (a) A cubical 3-polytope where $s_1$ is in Configuration 3F. (b) A facet of a cubical 5-polytope where $s_1$ is in Configuration 5F. }\label{fig:conf-Df} 
\end{figure}

\begin{lemma} Let $d\ge 5$ be odd and let $k:=(d+1)/2$. Let $s_{1}$ be a vertex in a cubical $d$-polytope and let $\St_{1}$  be the star of  $s_{1}$ in the polytope. Moreover, let $Y:=\{\{s_{1},t_{1}\},\ldots,\{s_{k},t_{k}\}\}$ be a labelling and pairing of $2k$ distinct vertices of $\St_{1}$.  Then the set $Y$ is linked in $\St_{1}$ if the vertex $s_{1}$ is not in Configuration $d$F.  
\label{lem:star-cubical} 
\end{lemma}

\begin{remark}
It is easy to see that when the vertex $s_1$ is in Configuration $d$F, the set $Y$ is not linked in $\St_{1}$. Indeed in this case, since $\dist_{F_{1}}(s_{1},t_{1})=d-1$ there is only one facet in $\St_{1}$ in $F_1$ that contains $t_1$.
 Then all the neighbours of $t_{1}$ in $F_{1}$, and thus, in $\St_{1}$ are in $X$. As a consequence, every $s_{1}-t_{1}$ path in $\St_{1}$ must touch $X$. Hence $Y$ is not linked.  
\end{remark}
	 
We defer the proof of \cref{lem:star-cubical} for $d\geq 7$ to Subsection~\ref{subsec:star-cubical}, while the case $d=5$ is proved in \Cref{app:lemmad5}. We are now ready to prove our main result, assuming \cref{lem:star-cubical}. For a set $Y:=\{\{s_{1},t_{1}\},\ldots,\{s_{k},t_{k}\}\}$ of pairs of vertices in a graph, a {\it $Y$-linkage} $\{L_{1},\ldots,L_{k}\}$ is a set of disjoint paths with the path $L_{i}$ joining the pair $\{s_{i},t_{i}\}$ for each $i\in [1,k]$.  For a path $L:=u_{0}\ldots u_{n}$ we often write $u_{i}Lu_{j}$ for $0\le i\le j\le n$  to denote the subpath $u_{i}\ldots u_{j}$. We will rely on the following definition.
	
\begin{definition}[Projection $\pi$]\label{def:projection}
	For a pair of opposite facets $\{F,F^{o}\}$ of $Q_{d}$, define a projection $\pi^{Q_{d}}_{F^{o}}$ from $Q_{d}$ to $F^{o}$ by sending a vertex $x\in F$ to the unique neighbour $x^{p}_{F^o}$ of $x$ in $F^{o}$, and  a vertex $x\in F^{o}$ to itself (that is, $\pi^{Q_{d}}_{F^{o}}(x)=x$); write $\pi^{Q_{d}}_{F^o}(x)=x^{p}_{F^o}$ to be precise, or write $\pi(x)$ or $x^{p}$ if the cube $Q_{d}$ and the facet $F^o$ are understood from the context. 
\end{definition}

We extend this projection to sets of vertices: given a pair $\{F,F^{o}\}$ of opposite facets and a set $X\subseteq V(F)$, the projection $X^{p}_{F^{o}}$ or $\pi^{Q_{d}}_{F^{o}}(X)$ of $X$ onto $F^{o}$ is the set of the projections  of the vertices in $X$ onto $F^{o}$. For an $i$-face $J\subseteq F$, the projection $J^{p}_{F^{o}}$ or $\pi_{F^{o}}^{Q_{d}}(J)$ of $J$ onto $F^{o}$ is the $i$-face consisting of the projections of all the vertices of $J$ onto $F^{o}$. For a pair $\{F,F^{o}\}$ of opposite facets in $Q^{d}$, the restrictions of the projection $\pi_{F^{o}}$ to $F$  and the projection $\pi_{F}$ to $F^{o}$ are bijections.

\begin{proof}[Proof of \cref{thm:cubical} (Linkedness of cubical polytopes)]

\cref{prop:weak-linkedness-cubical}  settled the case of even $d$, so we assume $d$ is odd.

Let $d$ be odd and $d\ge 5$ and let $k:=(d+1)/2$.  Let $X$ be any set of $2k$ vertices in the graph $G$ of a cubical $d$-polytope $P$. Recall the vertices in $X$ are called terminals. Also let $Y:=\{\{s_{1},t_{1}\},\ldots,\{s_{k},t_{k}\}\}$ be a labelling and pairing of the vertices of $X$. We aim to find a $Y$-linkage $\{L_{1},\ldots,L_{k}\}$ in $G$ where $L_{i}$ joins the pair $\{s_{i},t_{i}\}$ for $i=1,\ldots,k$. 

 For a set of vertices $X$ of a graph $G$, a path in $G$ is called {\it $X$-valid} if no inner vertex of the path is in $X$. The {\it distance} between two vertices $s$ and $t$ in $G$, denoted $\dist_{G}(s,t)$, is the length of a shortest path between the vertices.

The first step of the proof is to reduce the analysis space from the whole polytope to a more manageable space, the star $\St_1$ of a terminal vertex in the boundary complex of $P$, say that of $s_{1}$. We do so by considering $d=2k-1$ disjoint paths $S_{i}:=s_{i}-\St_1$ (for each $i\in [2,k]$) and $T_{j}:=t_{j}-\St_1$ (for each  $j\in [1,k]$) from the terminals into $\St_{1}$. Here we resort to the $d$-connectivity of $G$. In addition, let $S_{1}:=s_{1}$. We then denote by $\bar s_{i}$ and $\bar t_{j}$ the intersection of the paths $S_{i}$ and $T_{j}$ with $\St_1$. Using the vertices  $\bar s_{i}$ and $\bar t_{i}$ for $i\in [1,k]$, define sets $\bar X$ and $\bar Y$ in $\St_{1}$, counterparts to the sets $X$ and $Y$ of $G$. In an abuse of terminology, we also say that the vertices $\bar s_{i}$ and $\bar t_{i}$ are terminals. In this way, the existence of a $\bar Y$-linkage $\{\bar L_{1},\ldots,\bar L_{k}\}$ with $\bar L_{i}:=\bar s_{i}-\bar t_{i}$ in $G(\St_{1})$ implies the existence of a $Y$-linkage $\{L_{1},\ldots,L_{k}\}$ in $G(P)$, since each path $\bar L_{i}$ ($i\in [1,k]$) can be extended with the paths $S_{i}$ and $T_{i}$ to obtain the corresponding path $L_{i}=s_{i}S_{i}\bar s_{i}\bar L_{i}\bar t_{i}T_{i}t_{i}$.

The second step of the proof is to find a $\bar Y$-linkage $\{\bar L_{1},\ldots,\bar L_{k}\}$ in $G(\St_{1})$, whenever possible.  According to \cref{lem:star-cubical}, there is a $\bar Y$-linkage in $G(\St_{1})$ provided that the vertex $s_{1}$ is not in Configuration $d$F. The existence of a $\bar Y$-linkage in turn gives the existence of a $Y$-linkage, and completes the proof of the theorem in this case.

The third and final step is to deal with Configuration $d$F for $s_{1}$. Hence assume that the vertex $s_{1}$  is in Configuration $d$F. This is implies that
\begin{enumerate}	
 \item[(i)] there exists a unique facet $F_{1}$ of $\St_1$ containing $\bar t_{1}$; that 
 \item[(ii)]   $|\bar X\cap V(F_{1})|= d+1$; and that
  \item[(iii)]  $\dist_{F_{1}}(\bar s_{1},\bar t_{1})=d-1$ and all the $d-1$ neighbours of $\bar t_{1}$ in $F_{1}$, and thus in $\St_1$, belong to $\bar X$.
\end{enumerate}	

Let $R$ be a $(d-2)$-face of $F_{1}$ containing the vertex $s_1^o$ opposite to $s_1$ in $F_1$, then $s_{1}\not \in R$, and $\bar{t}_1 = s_1^o \in R$. Denote by $R_{F_{1}}$ the  $(d-2)$-face of $F_{1}$ disjoint from $R$. Let $J$ be the other facet of $P$ containing $R$ and let $R_{J}$ denote the  $(d-2)$-face of $J$ disjoint from $R$. Then $R_{J}$ is disjoint from $F_{1}$. Partition the vertex set $V(R_{J})$ of $R_{J}$ into the vertex sets of two induced subgraphs $G_{\text{bad}}$ and $G_{\text{good}}$ such that $G_{\text{bad}}$ contains the neighbours of the terminals in $R$, namely $V(G_{\text{bad}})=\pi_{R_{J}}^{J}(\bar X\cap V(R))$ and $V(G_{\text{good}})=V(R_{J})\setminus V(G_{\text{bad}})$. Then $\pi_{R}^{J}(V(G_{\text{bad}}))\subseteq \bar X$ and $\pi_{R}^{J}(V(G_{\text{good}}))\cap \bar X=\emptyset$.  See \cref{fig:Aux-Linked-Thm}(a).

Consider again the paths $S_{i}$ and $T_{j}$ that bring the vertices $s_{i}$ ($i\in [2,k]$) and $t_{j}$ ($j\in [1,k]$) into $\St_1$. Also recall that the paths $S_{i}$ and $T_{j}$ intersect $\St_1$ at $\bar s_{i}$ and $\bar t_{j}$, respectively. We distinguish two cases: either at least one path $S_{i}$ or $T_{j}$ touches $R_{J}$ or no  path  $S_{i}$ or $T_{j}$ touches $R_{J}$. In the former case we redirect one aforementioned path $S_{i}$ or $T_{j}$ to break Configuration $d$F for $s_{1}$ and use \cref{lem:star-cubical}, while in the latter case we find the $\bar Y$-linkage using the antistar of $s_{1}$.  
\begin{case} Suppose at least one path $S_{i}$ or $T_{j}$ touches $R_{J}$. \end{case}

 If possible, pick one such path, say $S_{\ell}$, for which it holds that $V(S_{\ell})\cap V(G_{\text{good}})\ne \emptyset$. Otherwise, pick one such path, say $S_{\ell}$, that does not contain $\pi_{R_{J}}^{J}(t_{1})$, if it is possible. If none of these two selections are possible, then there is exactly one path $S_{i}$ or $T_{j}$ touching $R_{J}$, say $S_{\ell}$, in  which  case  $\pi_{R_{J}}^{J}(t_{1})\in V(S_{\ell})$.

We replace the path $S_{\ell}$ by a new path $s_{\ell}-\St_{1}$ that is disjoint from the other paths $S_{i}$ and $T_{j}$ and we replace the old terminal $\bar s_{\ell}$ by a new terminal that causes $s_{1}$ not to be in Configuration $d$F. First suppose that there exists $s_{\ell}'$  in  $V(S_{\ell})\cap V(G_{\text{good}})$. Then the old  path $S_{\ell}$ is replaced by the path $s_{\ell}S_{\ell}s'_{\ell}\pi_{R}^{J}(s_{\ell}')$, and the old terminal $\bar s_{\ell}$ is replaced by $\pi_{R}^{J}(s_{\ell}')$. Now suppose that $V(S_{\ell})\cap V(G_{\text{good}})=\emptyset$. Then every  path $S_{i}$ and $T_{j}$ that touches $R_{J}$ is disjoint from $G_{\text{good}}$. Denote by $s_{\ell}'$ the first intersection of $S_{\ell}$ with $R_J$. Let $M_{\ell}$ be a shortest path in $R_{J}$ from $s_{\ell}'\in V(G_{\text{bad}})$ to a vertex  $s_{\ell}''\in V(G_{\text{good}})$.  By our selection of $S_{\ell}$ this path $M_{\ell}$ always exists and is disjoint from any $S_i$ for $i\neq \ell$. If $s_{\ell}''\in V(G_{\text{good}})\setminus V(\St_1)$ then the old  path $S_{\ell}$ is replaced by the path $s_{\ell}S_{\ell}s'_{\ell}M_{\ell}s''_{\ell}\pi_{R}^{J}(s_{\ell}'')$, and the old terminal $\bar s_{\ell}$ is replaced by $\pi_{R}^{J}(s_{\ell}'')$. If instead  $s_{\ell}''\in V(G_{\text{good}})\cap V(\St_1)$ then the old  path $S_{\ell}$ is replaced by the path $s_{\ell}S_{\ell}s'_{\ell}M_{\ell}s''_{\ell}$, and the old terminal $\bar s_{\ell}$ is replaced by $s_{\ell}''$. Refer to \cref{fig:Aux-Linked-Thm}(b) for a depiction of this case.

In any case, the replacement of the old vertex $\bar s_{\ell}$ with the new $\bar s_{\ell}$  forces $s_{1}$ out of Configuration $d$F, and we can apply \cref{lem:star-cubical} to find a $\bar Y$-linkage.  The case of $S_{\ell}$ being equal to $T_{1}$ requires a bit more explanation in order to make sure that the vertex  $s_{1}$ does not end up in a new configuration $d$F.  Let $\A_1$ be the antistar  of $F_{1}$ in $\St_1$.  The new vertex $\bar t_{1}$ is either in $F_{1}$ or in $\A_{1}$. If the new $\bar t_{1}$ is in $F_{1}$ then it is plain that $s_{1}$ is not in Configuration $d$F. If the new vertex $\bar t_{1}$  is in $\A_{1}$,  then a new facet $F_{1}$ containing $s_{1}$ and the new $\bar t_{1}$ cannot contain all the $d-1$ neighbours of the old $\bar t_{1}$ in the old $F_{1}$, since the intersection between the new and the old $F_{1}$ is at most $(d-2)$-dimensional and no $(d-2)$-dimensional face of the old $F_{1}$ contains all the $d-1$ neighbours of the old $\bar t_{1}$. This completes the proof of the case. 

\begin{case} For any $(d-2)$-face $R$ in $F_{1}$ that contains $\bar t_{1}$, the aforementioned ridge $R_{J}$ in the facet $J$ is disjoint from all the paths $S_{i}$ and $T_{j}$. \end{case}
 

There is a unique neighbour of $\bar t_{1}$ in $R_{F_{1}}$, say $\bar s_{k}$, while every other neighbour of $\bar t_{1}$ in $F_{1}$ is in $R$. Let $\bar X^{p}:=\pi_{R_{J}}^{J}(\bar X\setminus\{s_{1},\bar s_{k},\bar t_{k}\})$ and let $s_{1}^{pp}:=\pi_{R_{J}}^{J}(\pi_{R}^{F_{1}}(s_{1}))$. See \cref{fig:Aux-Linked-Thm}(c). The $d-1$ vertices in $\bar X^{p}\cup\{s_{1}^{pp}\}$ can be linked in $R_{J}$ (\cref{thm:cube-linkedness}) by a linkage $\{\bar L_{1}',\ldots,\bar L_{k-1}'\}$. Observe that, for the special case of $d=5$ where $R_{J}$ is a 3-cube,  the sequence $s_{1}^{pp},  \pi_{R_{J}}^{J}(\bar s_{2}), \pi_{R_{J}}^{J}(\bar t_{1}),\pi_{R_{J}}^{J}(\bar t_{2})$ cannot be in a 2-face in cyclic order, since $\dist_{R_{J}}(s_{1}^{pp},\pi_{R_{J}}^{J}(\bar t_{1}))=3$. The linkage $\{\bar L_{1}',\ldots,\bar L_{k-1}'\}$ together with the two-path $\bar L_{k}:=\bar s_{k}\pi_{R_{F_{1}}}^{F_{1}}(\bar t_{k})\bar t_{k}$ can be extended to a linkage $\{\bar L_{1},\ldots, \bar L_{k}\}$ given by  
 
 \[\bar L_{i}:=\begin{cases}
 s_{1}\pi_{R}^{F_{1}}(s_{1})s_{1}^{pp}\bar L_{1}'\pi_{R_{J}}^{J}(\bar t_{1})\bar t_{1},& \text{for $i=1$;}\\ 
 \bar s_{i}\pi_{R_{J}}^{J}(\bar s_{i})\bar L_{i}'\pi_{R_{J}}^{J}(\bar t_{i})\bar t_{i},& \text{for $i\in [2,k-1]$;}\\
 \bar s_{k}\pi_{R_{F_{1}}}^{F_{1}}(\bar t_{k})\bar t_{k},& \text{for $i=k$.}
 \end{cases}\] 

Concatenating the paths $S_{i}$ (for all $i\in [2,k]$) and $T_{j}$ (for all $j\in [1,k]$) with the linkage $\{\bar L_{1},\ldots, \bar L_{k}\}$ gives the desired $Y$-linkage. This completes the proof of the case, and with it the proof of the theorem. 
\end{proof}
	 
\begin{figure}  
\includegraphics{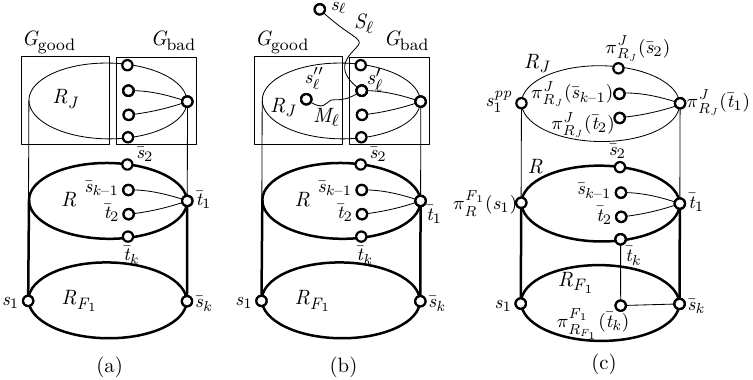}
\caption{Auxiliary figure for \cref{thm:cubical}, where the facet $F_{1}$ is highlighted in bold. (a) A depiction of the subgraphs $G_{\text{good}}$ and $G_{\text{bad}}$ of $R_{J}$. (b) A configuration where a path $S_{i}$ or $T_{j}$ touches $R_{J}$. (c) A configuration where no path $S_{i}$ or $T_{j}$ touches $R_{J}$.}\label{fig:Aux-Linked-Thm} 
\end{figure}

\subsection{Proof of \cref{lem:star-cubical} for $d\ge 7$}    
 \label{subsec:star-cubical}

Before starting the proof, we require several results.

\begin{proposition}[{\cite[Sec.~2]{Sal67}}]\label{prop:connected-complex-connectivity} For every $d\ge 1$, the graph of a strongly connected $d$-complex is $d$-connected. 
\end{proposition}

\begin{proposition}[{\cite[Prop. 27]{BuiPinUgo20a}}]
\label{prop:link-cubical} For every $d\ge 2$ such that $ d\ne 3$, the link of a vertex in a $(d+1)$-cube $Q_{d+1}$ is $\floor{(d+1)/2}$-linked.  
\end{proposition}

Let  $Z$ be a set of vertices in the graph of a $d$-cube $Q_{d}$. If, for some pair of opposite facets $\{F,F^o\}$, the set $Z$ contains both a vertex $z\in V (F)\cap Z$ and its projection $z_{F^{o}}^p\in V (F^o)\cap Z$, we say that the pair $\{F,F^o\}$ is {\it associated} with the set $Z$ in $Q_{d}$ and that $\{z,z^p\}$ is an {\it associating pair}.  Note that an associating pair can associate only one pair of opposite facets.

The next lemma lies at the core of our methodology. 
 
\begin{lemma}[{\cite[Lemma 8]{BuiPinUgo20a}}]\label{lem:facets-association} Let $Z$ be a nonempty subset of $V(Q_{d})$. Then the number of pairs $\{F,F^o\}$ of opposite facets associated with $Z$ is at most $|Z|-1$. 
\end{lemma} 	
	
The relevance of the lemma stems from the fact that a pair of opposite facets $\{F,F^{o}\}$ not associated with a given set of vertices $Z$ allows each vertex $z$ in $Z$ to have ``free projection''; that is, for every $z\in Z\cap V(F)$ the projection $\pi_{F^o}(z)$ is not in $Z$, and for $z\in Z\cap V(F^{o})$ the projection $\pi_{F}(z)$ is not in $Z$.	

\begin{lemma}[{\cite[Sec.~3]{WerWot11}}]\label{lem:k-linked-subgraph} Let $G$ be a $2k$-connected graph and let $G'$ be a $k$-linked subgraph of $G$. Then $G$ is $k$-linked.
\end{lemma}

\begin{proposition}\label{prop:star-minus-facet-linkedness} Let $F$ be a facet in the star $\St$  of a vertex in a cubical $d$-polytope. Then, for every $d\ge 2$, the antistar of $F$ in $\St$ is $\floor{(d-2)/2}$-linked. 
 \end{proposition}
 \begin{proof} Let $\St$ be the star of a vertex $s$ in a cubical $d$-polytope and let $F$ be a facet in the star $\St$. Let  $\mathcal A$ denote the antistar of $F$ in $\St$. 
 
The case of $d=2,3$ imposes no demand on $\mathcal A$, while the case $d=4,5$ amounts to establishing that the graph of $\mathcal A$ is connected. The graph of $\A$ is in fact $(d-2)$-connected, since $\mathcal A$ is a strongly connected $(d-2)$-complex (\cref{prop:star-minus-facet}). So assume $d\ge 6$. 
 
 There is a $(d-2)$-face $R$ in  $\A$. Indeed, take a $(d-2)$-face $R'$ in $F$ containing $s$ and consider the other facet $F'$ in $\St$ containing $R'$; the  $(d-2)$-face of $F'$ disjoint from $R'$ is the desired $R$.  By \cref{thm:cube-linkedness} the ridge $R$ is $\floor{(d-1)/2}$-linked but we only require it to be $\floor{(d-2)/2}$-linked. By \cref{prop:connected-complex-connectivity,prop:star-minus-facet}  the graph of $\A$ is $(d-2)$-connected.  Combining the linkedness of $R$ and the connectivity of the graph of $\A$ settles the proposition by virtue of \cref{lem:k-linked-subgraph}. 
 \end{proof}

For a pair of opposite facets $\{F,F^{o}\}$ in a cube, the restriction of the projection $\pi_{F^{o}}:Q_{d}\to F^{o}$ (\cref{def:projection}) to $F$ is a bijection from $V(F)$ to $V(F^{o})$. With the help of $\pi$, given the star $\St$ of a vertex $s$ in a cubical polytope and a facet $F$ in  $\St$, we can define an injection from the vertices in $F$,  except  the vertex opposite to $s$, to the antistar of $F$ in $\St$. Defining this injection is the purpose of \cref{lem:projections-star}.

 \begin{lemma}\label{lem:projections-star} Let $F$ be a facet in the star $\St$  of a vertex $s$ in a cubical $d$-polytope. Then there is an injective function, defined on the vertices of $F$ except the vertex $s^{o}$ opposite to $s$, that maps each such vertex in $F$ to a neighbour in $V(\St)\setminus V(F)$.  
 \end{lemma}  
\begin{proof} We construct the aforementioned injection $f$ between $V(F)\setminus \{s^{o}\}$ and $V(\St)\setminus V(F)$ as follows. Let $R_{1},\ldots, R_{d-1}$  be the $(d-2)$-faces of $F$ containing $s$, and let $J_{1},\ldots,J_{d-1}$ be the other facets of $\St$ containing $R_{1},\ldots, R_{d-1}$, respectively. Every vertex in $F$ other than $s^{o}$ lies in $R_{1}\cup\cdots\cup R_{d-1}$. Let $R_{i}^{o}$ be the $(d-2)$-face in $J_{i}$ that is opposite to $R_{i}$ for each $i\in[1,d-1]$. For every vertex $v$ in $V(R_{j})\setminus (V(R_{1})\cup \cdots \cup V(R_{j-1}))$ define  $f(v)$ as the projection $\pi$ in $J_{j}$ of $v$ onto $V(R_{j}^{o})$, namely $f(v):=\pi_{R^{o}_j}(v)$; observe that  $\pi_{R^{o}_j}(v)\in V(R_{j}^{o})\setminus (V(R_{1}^{o})\cup \cdots \cup V(R_{j-1}^{o}))$. Here $R_{-1}$ and $R_{-1}^{o}$ are empty sets. The function $f$ is well defined as $R_{i}$ and $R_{i}^{o}$ are opposite  $(d-2)$-cubes in the $(d-1)$-cube $J_{i}$.

To see that $f$ is an injection, take distinct vertices $v_{1},v_{2}\in V(F)\setminus \{s^{o}\}$, where $v_{1}\in V(R_{i})\setminus (V(R_{1})\cup \cdots \cup V(R_{i-1}))$ and $v_{2}\in V(R_{j})\setminus (V(R_{1})\cup \cdots \cup V(R_{j-1}))$ for $i\le j$. If $i=j$ then $f(v_{1})=\pi_{R^{o}_{i}}(v_{1})\ne \pi_{R^{o}_{i}}(v_{2})=f(v_{2})$. If instead $i< j$ then $f(v_{1})\in V(R_{i}^{o})\subseteq  V(R_{1}^{o})\cup \cdots \cup V(R_{j-1}^{o})$, while $f(v_{2})\not \in V(R_{1}^{o})\cup \cdots \cup V(R_{j-1}^{o})$. \end{proof}

\begin{proof}[Proof of \cref{lem:star-cubical} for $d\geq 7$]

The proof of the case $d=5$ follows a similar pattern to this one, but includes additional technical considerations due to the fact that the $3$-cube is not $2$-linked. These technical considerations will be presented in a separate proof in Appendix~\ref{app:lemmad5}.
In this proof, we identify the arguments that fail for $d=5$ with a dagger sign \dmark. This will make it easier for the reader to follow the proof for $d=5$ in the appendix.

Let $d\ge 7$ be odd and let $k:=(d+1)/2$.  Let $s_{1}$ be a vertex in a cubical $d$-polytope $P$ such that $s_{1}$ is not in Configuration $d$F, and let $\St_{1}$ denote the star of $s_{1}$ in $\B(P)$. Let $X$ be any set of $2k$ vertices in the graph $G(\St_{1})$ of  $\St_{1}$. The vertices in $X$ are our terminals. Also let $Y:=\{\{s_{1},t_{1}\},\ldots,\{s_{k},t_{k}\}\}$ be a labelling and pairing of the vertices of $X$. We aim to find a $Y$-linkage $\{L_{1},\ldots,L_{k}\}$ in $G$ where $L_{i}$ joins the pair $\{s_{i},t_{i}\}$ for $i=1,\ldots,k$. Recall that a path is $X$-valid if it contains no inner vertex from $X$.

We consider a facet $F_{1}$ of $\St_1$ containing $t_{1}$ and having the largest possible number of terminals.
We decompose the proof into four cases based on the number of terminals in $F_{1}$, proceeding from the more manageable case to the more involved one. 
\begin{enumerate}

\item[\cref{case:new-linkedness-thm-1}.] $|X\cap V(F_{1})|=d$.
\item[\cref{case:new-linkedness-thm-2}.] $3\le |X\cap V(F_{1})|\le d-1$.
\item[\cref{case:new-linkedness-thm-3}.] $|X\cap V(F_{1})|=2$ .
\item[\cref{case:new-linkedness-thm-4}.] $|X\cap V(F_{1})|=d+1$.  
 \end{enumerate}  

The proof of \cref{lem:star-cubical} is long, so we outline the main ideas. We let $\A_1$ be the antistar of $F_{1}$ in $\St_1$ and let $\Lk_1$ be the link of $s_{1}$ in $F_{1}$. Using the $(k-1)$-linkedness of $F_{1}$ (\cref{thm:cube-linkedness}), we link as many pairs of terminals in $F_{1}$ as possible through  disjoint $X$-valid paths $L_{i}:=s_{i}-t_{i}$. For those terminals that cannot be linked in $F_{1}$, if possible we use  the injection from $V(F_{1})$ to $V(\A_{1})$ granted by \cref{lem:projections-star} to find a set $N_{\A_{1}}$ of pairwise distinct neighbours  in $V(\A_{1})\setminus X$ of those terminals. Then, using the $(k-2)$-linkedness of $\A_{1}$ (\cref{prop:star-minus-facet-linkedness}), we link the corresponding pairs of terminals in $\A_{1}$ and vertices in $N_{\A_{1}}$  accordingly\dmark.  This general scheme does not always work, as the vertex $s_{1}^{o}$ opposite to $s_{1}$ in $F_{1}$ does not have an image in $\A_{1}$ under the aforementioned injection or the image of a vertex in $F_{1}$ under the injection may be a terminal. In those scenarios we resort to ad hoc methods, including linking corresponding pairs in the link of $s_{1}$ in $F_{1}$, which is $(k-1)$-linked by \cref{prop:link-cubical}\dmark and does not contain $s_{1}$ or $s_{1}^{o}$, or linking corresponding pairs in $(d-2)$-faces disjoint from $F_{1}$, which are $(k-1)$-linked by \cref{thm:cube-linkedness}\dmark.
 
To aid the reader, each case is broken down into subcases highlighted in bold.  

Recall that, given a pair $\{F,F^{o}\}$ of opposite facets in a cube $Q$,  for every vertex $z\in V (F)$ we denote by $z^{p}_{F^o}$ or  $\pi_{F^o}^{Q}(z)$ the unique neighbour of $z$ in $F^{o}$.

\begin{case}\label{case:new-linkedness-thm-1}  $|X\cap V(F_{1})|= d$.\end{case}
  
Without loss of generality,  assume that $  t_{2}\not \in V(F_{1})$. 

{\bf Suppose first that $\dist_{F_{1}} (  s_{2},s_{1})<d-1$}.  There exists a neighbour $  s_{2}'$ of $  s_{2}$ in $\A_1$. With the use of the strong $(k-1)$-linkedness of $F_{1}$ (\cref{thm:cube-strong-linkedness}), find disjoint paths $  L_{1}:=s_{1}-  t_{1}$ and $  L_{i}:=  s_{i}-  t_{i}$ (for each $i\in [3,k]$) in $F_{1}$, each avoiding $  s_{2}$. Find a path $  L_{2}$ in $\St_{1}$ between $  s_{2}$ and $  t_{2}$ that consists of the edge $  s_{2}  s_{2}'$ and a subpath in $\A_1$ between $  s_{2}'$ and $  t_{2}$, using the connectivity of $\A_1$ (see~\cref{prop:star-minus-facet}). The paths $  L_{i}$ ($i\in [1,k]$) give the desired $Y$-linkage. 

{\bf Now assume $\dist_{F_{1}}(  s_{2},s_{1})=d-1$}. Since $2k-1=d$ and there are $d-1$ pairs of opposite $(d-2)$-faces in $F_{1}$, by  \cref{lem:facets-association} there exists a pair $\{R,R^o\}$ of opposite $(d-2)$-faces in $F_{1}$ that is not associated with the set $  X_{s_{2}}:=(  X\cap V(F_{1}))\setminus \{  s_{2}\}$, whose cardinality is $d-1$. Assume $  s_{2}\in R$. Then $s_{1}\in R^{o}$.

Suppose all the neighbours of $  s_{2}$ in $R$ are in $   X$; that is, $N_{R}(  s_{2})=  X\setminus\{s_{1},  s_{2},  t_{2}\}$. The projection $\pi_{R^{o}}^{F_{1}}(  s_{2})$ of $  s_{2}$ onto $R^{o}$ is not in $  X$ since $s_{1}$ is the only terminal in $R^{o}$ and $\dist_{F_{1}}(  s_{2},s_{1})=d-1\ge 2$. Next find disjoint paths $  L_{i}:=  s_{i}-  t_{i}$ for all $i\in [3,k]$ in $R$ that do not touch $  s_{2}$ or $  t_{1}$, using the $(k-1)$-linkedness of $R$ (\emph{the argument also applies for $d=5$ due to the 3-connectivity of $R$ in this case}). With the help of \cref{lem:projections-star}, find a  neighbour   $s_{2}'$ of $\pi_{R^{o}}^{F_{1}}(  s_{2})$ in $\A_1$, and with the connectivity of $\A_1$, a path $  L_{2}$ between $  s_{2}$ and $  t_{2}$  that consists of the length-two path $  s_{2}\pi_{R^{o}}^{F_{1}}(  s_{2})s_{2}'$ and a subpath in $\A_1$ between  $s_{2}'$ and $  t_{2}$. Finally, find a path $  L_{1}$ in $F_{1}$ between $s_{1}$ and $  t_{1}$ that consists of the edge $  t_{1}\pi_{R^{o}}^{F_{1}}(  t_{1})$ and  a subpath in $R^{o}$ disjoint from $\pi_{R^{o}}^{F_{1}}(  s_{2})$  (here use the 2-connectivity of $R^{o}$).  The paths $  L_{i}$ ($i\in [1,k]$)  give the desired $Y$-linkage.
 
Thus assume there exists a neighbour $\bar  s_{2}$ of $  s_{2}$ in $V(R)\setminus   X$.  Let $ X_{R^{o}}:=\pi_{R^{o}}^{F_{1}}( X\setminus \{ s_{2}, t_{2}\})$. Find a path $L_2'$ in $\A_1$ between a neighbour $s'_2$ of $\bar{s}_2$ in $\A_1$ and $t_2$ using the connectivity of $\A_1$. Then let $L_2:=s_2\bar{s}_2s'_2L'_2t_2$.
Find disjoint paths $ L_{i}:=\pi_{R^{o}}^{F_{1}}( s_{i})-\pi_{R^{o}}^{F_{1}}( t_{i})$ ($i\in [1,k]$ and $i\ne 2$) in $R^{o}$ linking the $d-1$ vertices in $ X_{R^{o}}$ using the $(k-1)$-linkedness of $R^{o}$\dmark; add the  edge $\pi_{R^{o}}^{F_{1}}( t_{i}) t_{i}$ to $ L_{i}$ if $ t_{i}\in R$ or the  edge $\pi_{R^{o}}^{F_{1}}( s_{i}) s_{i}$ to $ L_{i}$ if $ s_{i}\in R$. The disjoint paths $ L_{i}$ ($i\in [1,k]$) gives the desired $Y$-linkage.

\begin{case}\label{case:new-linkedness-thm-2}  $3\le | X\cap V(F_{1})|\le d-1$.\end{case}
 The number of terminals in $\A_1$ is at most $d+1-3=d-2$.
 Since $2k-1=d$ and there are $d-1$ pairs of opposite $(d-2)$-faces in $F_{1}$, by \cref{lem:facets-association} there exists a pair $\{R,R^o\}$ of opposite $(d-2)$-faces in $F_{1}$ that is not associated with $ X\cap V(F_{1})$. Assume $s_{1}\in R$. We consider two subcases according to whether $ t_{1}\in R$ or $ t_{1}\in R^{o}$.

{\bf Suppose first that $ t_{1}\in R$}.  The $(d-2)$-connectivity of $R$ ensures the existence of an $X$-valid path $ L_{1}:=s_{1}- t_{1}$ in $R$. Let 
\[ X_{R^o}:=\pi_{R^{o}}^{F_{1}}(( X\setminus \{s_{1}, t_{1}\})\cap V(F_{1})).\] Then  $1\le | X_{R^o}|\le d-3$. Let $s_{1}^{o}$ be the vertex opposite to $s_{1}$ in $F_{1}$; the vertex $s_{1}^{o}$ has no neighbour in $\A_{1}$. 

Let $\bar Z$ be a set of $|V(\A_{1})\cap  X|$ distinct vertices in $V(R^o)\setminus ( X_{R^o}\cup \{s_{1}^{o}\})$. To see that $|\bar Z|\le |V(R^o)\setminus ( X_{R^o}\cup \{s_{1}^{o}\})|$, observe that, for $d\ge 5$ and   $| X_{R^o}|\le d-3$, we get \[|V(R^o)\setminus ( X_{R^o}\cup \{s_{1}^{o}\})|\ge 2^{d-2}-(d-3)-1\ge d-2\ge |V(\A_{1})\cap  X|= |\bar Z|.\] Use \cref{lem:projections-star} to obtain a set $Z$ in $\A_1$ of $|\bar Z|$ distinct vertices  adjacent to vertices in $\bar Z$. Then $|Z|=|V(\A_{1})\cap  X|\le d-2$.

Using the $(d-2)$-connectivity of $\A_1$ (\cref{prop:star-minus-facet}) and Menger's theorem, find disjoint paths $\bar S_{i}$ and $\bar T_{j}$ (for all $i,j\ne1$) in $\A_{1}$ between $V(\A_{1})\cap  X$ and $Z$. Then produce disjoint paths $S_{i}$ and $T_{j}$ (for all $i,j\ne1$) from terminals $s_{i}$ and $t_{j}$ in $\A_1$, respectively,  to $R^o$ by adding edges $z_{\ell}\bar z_{\ell}$ with $z_{\ell}\in Z$ and $\bar z_{\ell}\in \bar Z$ to the corresponding paths $\bar S_{i}$ and $\bar T_{j}$. If $ s_{i}$ or $ t_{j}$ is already in $R^o$, let $S_{i}:= s_{i}$ or $ T_{j}:= t_{j}$, accordingly. If instead $ s_{i}$ or $ t_{j}$ is in $R$, let $ S_{i}$ be the edge $ s_{i}\pi_{R^o}^{F_{1}}( s_{i})$ or let $ T_{j}$ be the edge $ t_{j}\pi_{R^o}^{F_{1}}( t_{j})$.  It follows that the paths $S_{i}$ and $ T_{i}$ for $i\in [2,k]$ are all pairwise disjoint. Let $ X_{R^o}^{+}$ be the intersections of $R^o$ and the paths $ S_{i}$ and $ T_{j}$ ($i,j\ne 1$). Then $| X_{R^o}^{+}|=d-1$.  Suppose that $X^{+}_{R^{o}}=\left\{\bar s_{2},\bar t_{2},\ldots, \bar s_{k},\bar t_{k}\right\}$. The corresponding pairing  $ Y_{R^o}^{+}$ of the vertices in $X_{R^o}^{+}$ can be linked through paths $\bar L_{i}:=\bar s_{i}-\bar t_{i}$ (for all $i\in [2,k]$) in $R^{o}$ using the $(k-1)$-linkedness of $R^o$ (\cref{thm:cube-linkedness})\dmark. See \cref{fig:Aux-Linked-Thm-NewCase2}(a) for a depiction of this configuration. In this case, the desired $Y$-linkage is given by the following paths.
 
 \[L_{i}:=\begin{cases}s_{1} L_{1} t_{1},& \text{for $i=1$;}\\s_{i}S_{i}\bar s_{i}\bar L_{i}\bar t_{i} T_{i} t_{i},& \text{otherwise.}
 \end{cases}\]   
    
\begin{figure}
\includegraphics{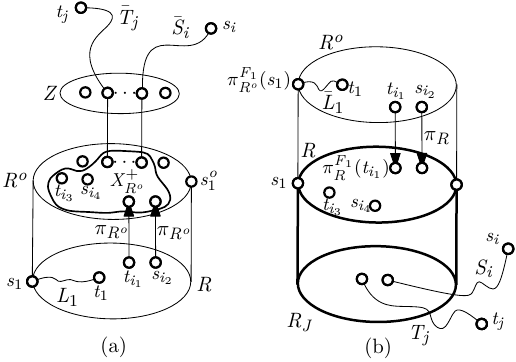}
\caption{Auxiliary figure for \cref{case:new-linkedness-thm-2} of \cref{lem:star-cubical}. (a) A configuration where $ t_{1}\in R$  and the subset $ X^{+}_{R^{o}}$ of $R^{o}$ is highlighted in bold. (b) A configuration where $ t_{1}\in R^{o}$  and the facet $J$ is highlighted in bold. }\label{fig:Aux-Linked-Thm-NewCase2} 
\end{figure}

{\bf Suppose now that $ t_{1}\in R^o$}.  Let \[ X_{R}:=\pi_{R}^{F_{1}}(( X\setminus\{ t_{1}\})\cap V(F_{1})).\] There are at most $d-2$ terminal vertices in $R^o$. Therefore, the $(d-2)$-connectivity of $R^o$ ensures the existence of an $ X$-valid $\pi_{R^{o}}^{F_{1}}(s_{1})- t_{1}$ path $\bar L_{1}$ in $R^o$. Then let $ L_{1}:=s_{1}\pi_{R^{o}}^{F_{1}}(s_{1})\bar L_{1} t_{1}$.  Let  $J$ be the other facet in $\St_1$ containing $R$ and let $R_{J}$ be the $(d-2)$-face of $J$ disjoint from $R$. Then $R_{J}\subset \A_{1}$. Since there are at most $d-2$ terminals in $\A_1$ and since $\A_1$ is $(d-2)$-connected (\cref{prop:star-minus-facet}), we can find corresponding disjoint paths $ S_{i}$ and $ T_{j}$ from the terminals in $\A_{1}$ to $R_{J}$ by Menger's theorem \cite[Theorem 3.3.1]{Die05}.  For terminals $ s_{i}$ and $ t_{j}$ in $ X\cap V(R)$, let $ S_{i}:= s_{i}$ and  $ T_{j}:= t_{j}$ for all $i,j\ne 1$, while for terminals $ s_{i}$ and $ t_{j}$ in $ X\cap V(R^o)$, let $ S_{i}:= s_{i}\pi_{R}^{F_{1}}( s_{i})$ and $ T_{j}:= t_{j}\pi_{R}^{F_{1}}( t_{j})$ for all $i,j\ne 1$. Let $ X_{J}$ be the set of the intersections of the paths $S_{i}$ and $ T_{j}$ with $J$ plus the vertex $s_{1}$. Then $ X_{J}\subset V(J)$ and $| X_{J}|=d$ (since $ t_{1}\ \in R^{o}$). Suppose that $X_{J}=\left\{s_{1},\bar s_{2},\bar t_{2},\ldots, \bar s_{k},\bar t_{k}\right\}$ and let $Y_{J}=\left\{\left\{\bar s_{2},\bar t_{2}\right\},\ldots, \left\{\bar s_{k},\bar t_{k}\right\}\right\}$ be a pairing of $X_{J}\setminus\left\{s_{1}\right\}$. 

Resorting to the strong $(k-1)$-linkedness of the facet $J$ (\cref{thm:cube-strong-linkedness}), we obtain $k-1$ disjoint paths $\bar L_{i}:=\bar s_{i}-\bar t_{i}$ for all $i\ne 1$ that correspondingly link $Y_{J}$ in $J$, with all the paths avoiding $s_{1}$. See \cref{fig:Aux-Linked-Thm-NewCase2}(b) for a depiction of this configuration. In this case, the desired $Y$-linkage is given by the following paths.
\[L_{i}:=\begin{cases}s_{1}L_{1} t_{1},& \text{for $i=1$;}\\s_{i}S_{i} \bar L_{i} T_{i} t_{i},& \text{otherwise.}
 \end{cases}\] 	  
  
\begin{case}\label{case:new-linkedness-thm-3}  $| X\cap V(F_{1})|=2$.\end{case}
  
In this case, we have that $V(F_1)\cap X = \{s_1,t_1\}$ and $|V(\A_1)\cap  X|=d-1$. The proof of this case requires the definition of several sets. For quick reference, we place most of these definitions in itemised lists. We begin with the following sets:

\begin{itemize}
\item $\St_{12}$, the star of $ s_{2}$ in $\St_{1}$ (that is, the complex formed by the facets of $P$ containing $s_{1}$ and $ s_{2}$);
\item $G(\St_{12})$, the graph of $\St_{12}$; and 
\item $\Gamma_{12}$, the subgraph of $G(\St_{12})$ and $G(\A_1)$ that is induced by  $V(\St_{12})\setminus V(F_{1})$.
\end{itemize}
It follows that every neighbour in $G(\A_{1})$ of $s_{2}$ is in $\Gamma_{12}$:
\begin{equation}\label{eq:neighbourhood}
	N_{\Gamma_{12}}(s_{2})=N_{G(\A_{1})}(s_{2}). 
\end{equation}
Note that when $d\geq 5$, $|V(\Gamma_{12})| \geq 2^{d-2}\geq d-2$, since $\St_{12}$ contains at least one facet (other than $F_1$), and that facet contains at least one $(d-2)$-face disjoint from $F_1$. The vertices of that $(d-2)$-face are in $\Gamma_{12}$.

{\bf The first step for this case is to bring the terminals in  $\A_{1}$ into $\Gamma_{12}$.} The $(d-2)$-connectivity of the graph $G(\A_1)$ (\Cref{prop:star-minus-facet}) ensures the existence of pairwise disjoint paths from $(V(\A_1)\cap X)\setminus \{s_2\}$ to $\Gamma_{12}$. Among these paths, denote by $S_{i}$ the path from the terminal $s_{i}\in \A_1$ to  $\Gamma_{12}$ and let $V(S_{i})\cap V(\Gamma_{12})=\left\{\hat s_{i}\right\}$. Similarly, define $ T_{j}$ and $\hat t_{j}$. By \eqref{eq:neighbourhood} each path $ S_{i}$ or $ T_{j}$ touches $\Gamma_{12}$ at a vertex other than $ s_{2}$; this is so because each such path will need to reach the neighbourhood of $s_{2}$ in $\Gamma_{12}$ before reaching $s_{2}$. We also let $\hat s_{2}$ denote $ s_{2}$. The set of vertices $\hat x$ is accordingly denoted by $\hat X$. Then $|\hat X|=d-1$. Abusing terminology, since there is no potential for confusion, we  call the vertices in $\hat X$ terminals as well. \Cref{fig:Aux-Linked-Thm-Case4}(a) depicts this configuration. 
 
Pick a facet  $F_{12}$ in $\St_{12}$ that contains $\hat t_{2}$.
An important point is that $ t_{1}$ is not in $F_{12}$; otherwise $F_{12}$ would contain $s_{1}$,$s_{2}$ and  $t_{1}$, and it should have been chosen instead of $F_{1}$.   
 
{\bf The second step is to find a path $L_{1}$ in $F_{1}$ between $s_{1}$ and $ t_{1}$ such that $V(L_{1})\cap V(F_{12})=\left\{s_{1}\right\}$.}

\begin{remark}
\label{rmk:two-vertices}
 For any two faces $F,J$ of a polytope, with $F$ not contained in $J$, there is a facet containing $J$ but not $F$. In particular, for any two distinct vertices of a polytope, there is a facet containing one but not the other.
\end{remark}
 
 To see the existence of such a path, note that the intersection of $F_{12}$ and $F_{1}$ is 
 a face that does not contain $t_{1}$ and therefore is
 contained in a $(d-2)$-face $R$  of $F_{1}$ containing $s_{1}$ but not $t_{1}$ (\cref{rmk:two-vertices}). Find a path $ L_{1}'$ in $R^{o}$, the $(d-2)$-face in $F_{1}$ disjoint from $R$ ($R^{o}$ contains $t_{1}$), between $\pi_{R^{o}}^{F_{1}}(s_{1})$ and $ t_{1}$ and let $ L_{1}:=s_{1}\pi_{R^{o}}^{F_{1}}(s_{1}) L_{1}' t_{1}$.
 
{\bf The third step is to bring the $d-1$ terminal vertices $\hat x\in \Gamma_{12}$ into the facet $F_{12}$ so that they can be linked there, avoiding $s_{1}$.}  We consider two cases depending on the number of facets in $\St_{12}$.
  
\begin{figure} 
\includegraphics[scale=.9]{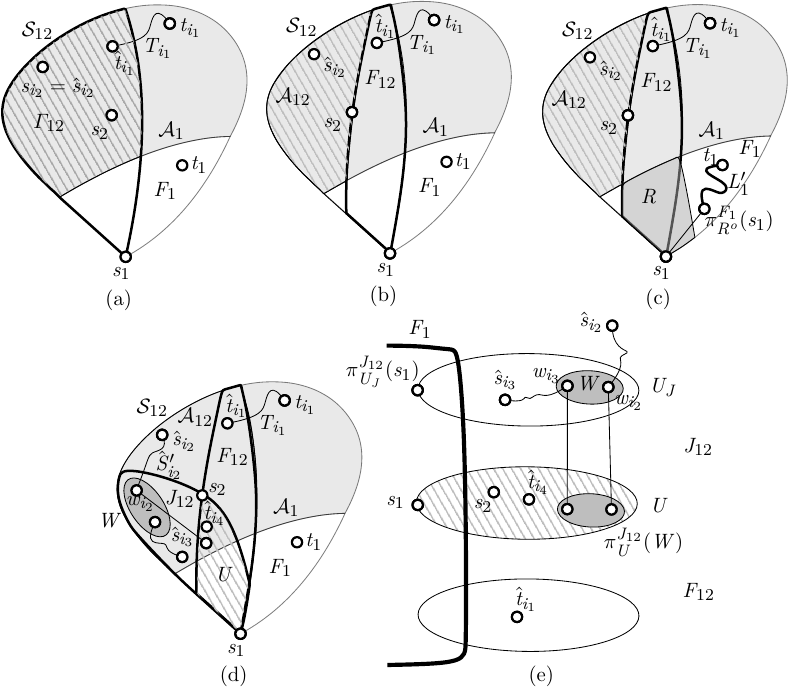}
\caption{Auxiliary figure for \cref{case:new-linkedness-thm-3} of \cref{lem:star-cubical}. A representation of $\St_{1}$. (a) A configuration where the subgraph $\Gamma_{12}$ is tiled in  falling pattern and the complex $\A_{1}$ is coloured in grey. (b) A depiction of $\St_{12}$ with more than one facet; the facet $F_{12}$ is highlighted in bold, the complex $\A_{1}$ is coloured in grey and the complex $\A_{12}$ is highlighted in falling pattern. (c) The construction of the path $L_1:=s_1\pi_{R^o}^{F^1}(s_1)L_1't_1$ from $s_1$ to $t_1$ in $F_1$ such that $L_1\cap V(\Gamma_{12}) = \{s_1\}$. (d)  A depiction of $\St_{12}$ with more than one facet; the facets  $F_{12}$ and $J_{12}$ are highlighted in bold and their intersection $U$ is highlighted in falling pattern; the set $W$ in $J_{12}$ is coloured in dark grey. (e) A depiction of a portion of $\St_{12}$, zooming in on the facets  $F_{12}$ and $J_{12}$; each facet is represented as the convex hull of two disjoint $(d-2)$-faces, and their intersection $U$  is highlighted in falling pattern. The sets $W$ and $\pi_{U}^{J_{12}}(W)$ in $J_{12}$ are coloured in dark grey.}\label{fig:Aux-Linked-Thm-Case4} 
\end{figure}
   
{\bf Suppose $\St_{12}$ only consists of $F_{12}$}. Then \[\hat X=\{\hat s_{2},\ldots,\hat s_{k},\hat t_{2},\ldots,\hat t_{k}\}\subset V(\Gamma_{12})\subset V(F_{12}).\] With the help of the strong $(k-1)$-linkedness of $F_{12}$ (\cref{thm:cube-strong-linkedness}), we can link the pair $\{\hat s_{i},\hat t_{i}\}$ for each $i\in [2,k]$ in $F_{12}$ through disjoint paths $\hat L_{i}$, all avoiding $s_{1}$.  The path $\hat L_{i}$ concatenated with the paths $S_{i}$ and $ T_{i}$  for each $i\in [2,k]$ give a $(Y\setminus \{s_{1}, t_{1}\})$-linkage $\{L_{2},\ldots,L_{k}\}$. Hence the desired $Y$-linkage is as follows.  
 
\[L_{i}:=\begin{cases}s_{1}\pi_{R^{o}}^{F_{1}}(s_{1}) L_{1}' t_{1},& \text{for $i=1$;}\\s_{i}S_{i}\hat s_{i}\hat L_{i}\hat t_{i} T_{i}t_{i},& \text{otherwise.}
 \end{cases}\]
  
{\bf Assume $\St_{12}$  has more than one facet}. We have that \[\hat X=\{\hat s_{2},\ldots,\hat s_{k},\hat t_{2},\ldots,\hat t_{k}\}\subset V(\Gamma_{12}).\] 
Define
\begin{itemize}
\item $\A_{12}$ as the complex of $\St_{12}$ induced by $V(\St_{12})\setminus (V(F_{1})\cup V(F_{12}))$. 
\end{itemize}
Then the graph $G(\A_{12})$ of $\A_{12}$  coincides with the subgraph of $\Gamma_{12}$ induced by $V(\Gamma_{12})\setminus V(F_{12})$. \Cref{fig:Aux-Linked-Thm-Case4}(b) depicts this configuration.

Our strategy is first to bring the $d-3$ terminal vertices $\hat x$ in $\Gamma_{12}$ other than $\hat s_{2}$ and $\hat t_{2}$ into $F_{12}\setminus F_{1}$ through disjoint paths $\hat S_{i}$ and $\hat T_{j}$, without touching $\hat s_{2}$ and $\hat t_{2}$. Second, denoting by $\tilde s_{i}$ and $\tilde t_{j}$ the intersection of $\hat S_{i}$ and $\hat T_{j}$ with $V(F_{12})\setminus V(F_{1})$, respectively, we link the pairs $\{\tilde s_{i},\tilde t_{i}\}$ for all $i\in[2,k]$ in $F_{12}$ through disjoint paths $\tilde L_{i}$, without touching $s_{1}$; here we resort to the strong $(k-1)$-linkedness of $F_{12}$. We develop these ideas below. 

From  \cref{lem:technical}(iii), it follows that $\A_{12}$ is nonempty and contains a spanning strongly connected $(d-3)$-subcomplex, thereby implying, by \cref{prop:connected-complex-connectivity}, that \[\text{$G(\A_{12})$ is $(d-3)$-connected.}\]
 Since $\St_{12}$ contains more than one facet, the following sets exist:
 \begin{itemize}
 \item $U$, a $(d-2)$-face  in $F_{12}$ that contains $s_{1}$ and $\hat s_{2}$ ($= s_{2}$)  (since several facets in $\St_{12}$ contain both $s_1$ and $s_2$);
 \item $J_{12}$, the other facet in $\St_{12}$ containing $U$;
 \item $U_{J}$, the $(d-2)$-face in $J_{12}$ disjoint from $U$, and as a consequence, disjoint from $F_{12}$; 
 \item $\C_{U}$,  the subcomplex of $\B(U)$ induced by $V(U)\setminus V(F_{1})$, namely the antistar of $U\cap F_{1}$ in $U$; and 
 
\item  $\C_{U_{J}}$, the subcomplex of $\B(U_{J})$ induced by $V(U_{J})\setminus  V(F_{1})$. 
 \end{itemize}
	
 The subcomplex $\C_{U}$ is nonempty, since $\hat s_{2}\in V(U)\setminus V(F_{1})$, and so, thanks to \cref{lem:cube-face-complex}, it is a strongly connected $(d-3)$-complex. Then, from $C_{U}$ containing a $(d-3)$-face it follows that
 \begin{equation}\label{eq:Cubical-Linkedness-Case4-U-Cardinality}
 |V(\C_{U})|=|V(U)\setminus V(F_{1}))|\ge 2^{d-3}\ge d-1\; \text{for $d\ge 5$}.
 \end{equation}
 
The subcomplex $\C_{U_{J}}$ is nonempty: 
the vertex in $J_{12}$ opposite to $s_{1}$ is not in $U$, since $s_{1}\in U$, nor is it in $F_{1}$ (\cref{rmk:opposite-vertex-1}), and so it must be in $\C_{U_{J}}$. 
If $U_{J}\cap F_{1}=\emptyset$ then $\C_{U_{J}}=\B(U_{J})$; otherwise $\C_{U_{J}}$ is the antistar of $U_{J}\cap F_{1}$ in $U_{J}$, and since $U\cap F_{1}\ne \emptyset$ ($s_{1}$ is in both), it follows that $U_{J}\not \subseteq F_{1}$. 
Therefore, according to \cref{lem:cube-face-complex}, $\C_{U_{J}}$ is or contains a strongly connected $(d-3)$-complex.   Hence, in both instances,
 \begin{equation}\label{eq:Cubical-Linkedness-Case4-U_J-Cardinality}
  |V(\C_{U_{J}})|=|V(U_{J})\setminus V(F_{1}))|\ge 2^{d-3}\ge d-1 \; \text{for $d\ge 5$}.
 \end{equation}
 
Recall that we want to bring every vertex in the set $\hat X$, which is contained in $\Gamma_{12}$, into $F_{12}\setminus F_{1}$. We construct $|\hat X\cap V(\A_{12})|$ pairwise disjoint paths $\hat S_{i}$ and $\hat T_{j}$ from $\hat s_{i}\in \A_{12}$ and $\hat t_{j}\in \A_{12}$, respectively,  to $V(F_{12})\setminus V(F_{1})$ as follows.  Pick a set \[W\subset V(\C_{U_{J}})\setminus \pi_{U_{J}}^{J_{12}}\left((\hat X\cup\{s_{1}\})\cap U\right)\]
of $|\hat X\cap V(\A_{12})|$ vertices in $\C_{U_{J}}$. Then $\pi_{U}^{{J^{12}}}(W)$ is disjoint from $(\hat X\cup\{s_{1}\})\cap U$. In other words, the vertices in $W$ are in $\C_{U_{J}}$ and are not projections  of the vertices in $(\hat X\cup\{s_{1}\}) \cap U$ onto $U_{J}$.  We show that the set $W$ exists, which amounts to showing that $\C_{U_{J}}$ has enough vertices to accommodate $W$.

First note that 
\begin{equation}
\begin{gathered}\label{eq:Cardinality-hat-X-A}
|\hat X\cap V(\A_{12})|+|(\hat X\cup\{s_{1}\})\cap V(F_{12})|=|\hat X\cup \{s_{1}\}|=d,\\ 
(\hat X\cup\{s_{1}\})\cap V(U)\subseteq (\hat X\cup\{s_{1}\})\cap V(F_{12}).
\end{gathered}  
\end{equation} 

If $U_{J}\cap F_{1}=\emptyset$ then $\C_{U_{J}}=\B(U_{J})$. And \eqref{eq:Cardinality-hat-X-A} together with $|V(U_{J})|=2^{d-2}\ge d$ for $d\ge 7$ (indeed, for $d\ge 5$) gives the following chain of inequalities
\begin{multline*}
\left|V(C_{U_{J}})\setminus \pi_{U_{J}}^{J_{12}}\left((\hat X\cup\{s_{1}\})\cap V(U)\right)\right|\ge \left| V(U_J)\right|-\left|(\hat X\cup \{s_{1}\})\cap V(U)\right|\\
\ge d-\left|(\hat X\cup \{s_{1}\})\cap V(U)\right|
\ge\left|\hat X\cup\{s_{1}\}\right|-\left|(\hat X\cup \{s_{1}\})\cap V(F_{12})\right|\\
=\left|\hat X\cap V(\A_{12})\right|=\left|W\right|,
\end{multline*}
as desired. 

Suppose now $U_{J}\cap F_{1}\ne \emptyset$. Since $s_{1}\in U\cap F_{1}$ and $J_{12}=\conv \{U\cup U_{J}\}$, the cube $J_{12}\cap F_{1}$ has opposite facets $U_{J}\cap F_{1}$ and $U\cap F_{1}$. From $s_{1}\in U\cap F_{1}$ it follows that $\pi_{U_{J}}^{J_{12}}(s_{1})\in U_{J}\cap F_{1}$, and thus, that $\pi_{U_{J}}^{J_{12}}(s_{1})\not \in \C_{U_{J}}$; here we use the following remark.
\begin{remark}
Let $(K,K^{o})$ be opposite facets in a cube $Q$ and let $B$ be a proper face of $Q$ such that $B\cap K\ne \emptyset$ and $B\cap K^{o}\ne \emptyset$. Then $\pi^{Q}_{K^{o}}(B\cap K)=B\cap K^{o}$. 
\end{remark}
\noindent Since $\pi_{U_{J}}^{J_{12}}(s_{1})\not \in \C_{U_{J}}$, using \eqref{eq:Cubical-Linkedness-Case4-U_J-Cardinality} and \eqref{eq:Cardinality-hat-X-A} we get  	   
\begin{multline*} 
\left|V(C_{U_{J}})\setminus \pi_{U_{J}}^{J_{12}}\left((\hat X\cup\{s_{1}\})\cap V(U)\right)\right|=\left|V(C_{U_{J}})\setminus \pi_{U_{J}}^{J_{12}}\left(\hat X\cap V(U)\right)\right|
\\ \ge \left|V(C_{U_J})\right|-\left|\hat X\cap V(U)\right|
 \ge d-1-\left|\hat X\cap V(U)\right|\\
\ge\left|\hat X\right|-\left|\hat X\cap V(F_{12})\right|
=\left|\hat X\cap V(\A_{12})\right|=\left|W\right|.
\end{multline*}
{\bf In this way, we have shown that $\C_{U_{J}}$ can accommodate the set $W$.} We now finalise the case.
  
There are at most $d-3$ vertices $\hat x$ in $\hat X\cap V(\A_{12})$ because $\hat s_{2}$ and $\hat t_{2}$ are already in $V(F_{12})\setminus V(F_{1})$. Since $G(\A_{12})$ is $(d-3)$-connected, we can find $|W|=|\hat X\cap V(\A_{12})|$ pairwise disjoint paths $\hat S_{i}'$ and $\hat T_{j}'$ in $\A_{12}$ from the terminals $\hat s_{i}$ and $\hat t_{j}$ in $\hat X\cap V(\A_{12})$ to $W$.  The $\hat X$-valid path $\hat S_{i}$ from $\hat s_{i}\in \A_{12}$ to $V(F_{12})\setminus V(F_{1})$ then consists of the subpath $\hat S_{i}':=\hat s_{i}-w_{i}$ with $w_{i}\in W$  plus the edge $w_{i}\pi_{U}^{J_{12}}(w_{i})$; from the choice of $W$ it follows that $\pi_{U}^{J_{12}}(w_{i})\not \in \hat X\cup\{s_{1}\}$. The paths $\hat T_{j}'$ and $\hat T_{j}$ are defined analogously. \Cref{fig:Aux-Linked-Thm-Case4}(d)-(e) depicts this configuration. 

Denote by $\tilde s_{i}$ the intersection of $\hat S_{i}$ and $V(F_{12})\setminus V(F_{1})$; similarly, define  $\tilde t_{j}$. Every terminal vertex $\hat x$ already in $F_{12}$ is also denoted by $\tilde x$, and in this case we let $\hat S_{i}$ or $\hat T_{j}$ be the vertex $\tilde x$.  
 
Now  $F_{12}$ contains the pairs $\left\{\tilde s_{i},\tilde t_{i}\right\}$ for all $i\in [2,k]$ and the terminal $s_{1}$, as desired. Link these pairs in $F_{12}$ through disjoint paths $\tilde L_{i}$, each avoiding $s_{1}$, with the use of the strong $(k-1)$-linkedness of $F_{12}$ (\cref{thm:cube-strong-linkedness}). The paths $\tilde L_{i}$ concatenated with the paths $S_{i}$, $\hat S_{i}$, $T_{i}$ and $\hat T_{i}$ for $i\in [2,k]$ give a $(Y\setminus \{s_{1}, t_{1}\})$-linkage $\{L_{2},\ldots,L_{k}\}$. Hence the desired $Y$-linkage is as follows.   
 
\[L_{i}:=\begin{cases}s_{1}\pi_{R^{o}}^{F_{1}}(s_{1}) L_{1}' t_{1},& \text{for $i=1$;}\\s_{i}S_{i}\hat s_{i}\hat S_{i}\tilde s_{i}\tilde L_{i}\tilde t_{i}\hat T_{i}\hat t_{i} T_{i} t_{i},& \text{otherwise.}
 \end{cases}\]

\begin{case}\label{case:new-linkedness-thm-4} $| X\cap V(F_{1})|= d+1$.\end{case}

Remember that by assumption $s_1$ is not in configuration $d$F. Here we have that $V(\A_{1})\cap  X=\emptyset$. This case is decomposed into three main subcases A, B and C, based on the nature of the vertex $s_{1}^{o}$ opposite to $s_{1}$ in $F_{1}$, which is the only vertex in $F_{1}$ that does not have an image under the injection from $F_{1}$ to $\A_{1}$ defined in \cref{lem:projections-star}. 

\subsection *{\bf \uppercase{Subcase} A. The vertex $s_{1}^{o}$ opposite  to $s_{1}$ in $F_{1}$ does not belong to $ X$.} Let $ X':= X\setminus \{ t_{1}\}$ and let $ Y':= Y\setminus \{\{s_{1}, t_{1}\}\}$. Since $| X'|=d$,  the strong $(k-1)$-linkedness of $F_{1}$ (\cref{thm:cube-strong-linkedness}) gives a $ Y'$-linkage $\{ L_{2},\ldots, L_{k}\}$ in the facet $F_{1}$ with each path $ L_{i}:= s_{i}- t_{i}$ ($i\in [2,k]$) avoiding $s_{1}$.  We find pairwise distinct  neighbours $s_{1}'$ and $ t_{1}'$ in $\A_1$ of $s_{1}$ and $ t_{1}$, respectively. If none of the paths $ L_{i}$ touches $ t_{1}$, we find a path $ L_{1}:=s_{1}- t_{1}$ in $\St_{1}$ that contains a subpath in $\A_1$  between  $s_{1}'$ and $ t_{1}'$ (here use the connectivity of $\A_{1}$, \cref{prop:star-minus-facet}), and we are home. Otherwise,  assume that the path $ L_{j}$ contains $ t_{1}$. With the help of \cref{lem:projections-star}, find pairwise distinct  neighbours $ s_{j}'$ and $ t_{j}'$ in $\A_1$ of $ s_{j}$ and $ t_{j}$, respectively, such that the vertices $s_{1}'$,  $ t_{1}'$, $ s_{j}'$ and $ t_{j}'$ are pairwise distinct. According to \cref{prop:star-minus-facet-linkedness}, the complex $\A_1$ is 2-linked for $d\ge 7$\dmark. Hence, we can find disjoint paths $ L_{1}':=s_{1}'- t_{1}'$ and $ L_{j}':= s_{j}'- t_{j}'$ in $\A_{1}$, respectively; these paths naturally give rise to paths $ L_{1}:=s_{1}s_{1}'L_{1}'t_{1}'t_{1}$ in $\St_1$  and $ L_{j}:= s_{j}s_{j}'L_{j}'t_{j}'t_{j}$ in $\St_1$. The paths $\left\{L_{1},\ldots,L_{k}\right\}$ give the desired $Y$-linkage.

\subsection *{\bf \uppercase{Subcase} B. The vertex $s_{1}^{o}$ opposite  to $s_{1}$ in $F_{1}$ belongs to $ X$ but is different from $ t_{1}$, say $s_{1}^{o}= s_{2}$.}  Since $F_1$ is a cube, the link $\Lk_1$ of $s_{1}$ in $F_{1}$ contains  all the vertices in $F_{1}$ except $s_{1}$ and $ s_{2}$. First find a neighbour $s_{1}'$ of $s_{1}$ and a neighbour $ t_{1}'$ of $ t_{1}$ in $\A_1$. 
 There is a neighbour $ s_{2}^{F_{1}}$ of $ s_{2}$ in $F_{1}$ that is either $ t_{2}$ or a vertex not in $ X$: $\{s_{1}, s_{2}\}\cap N_{F_{1}}( s_{2})=\emptyset$ and $|N_{F_{1}}( s_{2})|=d-1$. 

Suppose $ s_{2}^{F_{1}}= t_{2}$, and let $ L_{2}:= s_{2} t_{2}$. Using the $(k-1)$-linkedness of $\Lk_1$ (\cref{prop:link-cubical}), we find  disjoint paths  $ t_{1}- t_{2}$ and $ L_{i}:= s_{i}- t_{i}$  for each $i\in [3,k]$ in $\Lk_1$\dmark. Then define a path $ L_{1}:=s_{1}- t_{1}$ in $\St_{1}$ that contains a subpath in $\A_1$  between $s_{1}'$ and $ t_{1}'$; here we use the connectivity of $\A_{1}$ (\cref{prop:star-minus-facet}).  The paths $\left\{L_{1},\ldots,L_{k}\right\}$ give the desired $Y$-linkage. 

Assume $ s_{2}^{F_{1}}$ is not in $ X$. Observe that $| (X\setminus \{s_{1}, s_{2}\})\cup \{ s_{2}^{F_{1}}\}|=d$. Using the  $(k-1)$-linkedness of $\Lk_1$ for $d\ge 7$ (\cref{prop:link-cubical}), find in $\Lk_1$ disjoint paths  $ L_{2}':= s_{2}^{F_{1}}- t_{2}$ and $ L_{i}':= s_{i}- t_{i}$  for $i\in [3,k]$\dmark. Since $ t_{1}$ is also in $\Lk_{1}$ it may happen that it lies in one of the  paths $ L_{i}'$. If $ t_{1}$ does not belong to any of the paths $ L_{i}'$  for $i\in [2,k]$, then find a path $ L_{1}:=s_{1} s_{1}'L'_{1} t'_{1}t_{1}$ in $\St_{1}$ where $L_{1}'$ is a subpath in $\A_1$  between $s_{1}'$ and $ t_{1}'$, using the connectivity of $\A_{1}$ (\cref{prop:star-minus-facet}). In this scenario, let $L_{2}:= s_{2} s_{2}^{F_{1}} L_{2}' t_{2}$ and $L_{i}:=L_{i}'$ for each $i\in [3,k]$; the desired $Y$-linkage is given by the paths $\left\{L_{1},\ldots,L_{k}\right\}$.

If $ t_{1}$ belongs to one of the paths $ L_{i}'$ with $i\in [2,k]$, say $ L_{j}'$, then consider in $\A_1$ a neighbour $ t_{j}'$ of $ t_{j}$ and, either a neighbour $ s_{j}'$ of $ s_{j}$ if $j\ne 2$ or a neighbour $ s_{2}'$ of $ s_{2}^{F_{1}}$. From \cref{lem:projections-star} it follows that the vertices $s_{1}'$, $ t_{1}'$,  $ s_{j}'$ and $ t_{j}'$ can be taken pairwise distinct. Since $\A_1$ is 2-linked for $d\ge 7$ (see~\cref{prop:star-minus-facet-linkedness}), find in $\A_1$ a path $ L_{1}'$   between $s_{1}'$ and $ t_{1}'$ and a path $ L_{j}''$ between $ s_{j}'$ and $ t_{j}'$\dmark. As a consequence, we obtain in $\St_{1}$ a path  $ L_{1}:=s_{1}s_{1}' L_{1}' t_{1}' t_{1}$ and, either a path $ L_{j}:= s_{j} s_{j}' L_{j}'' t_{j}' t_{j}$ if $j\ne 2$ or a  path $ L_{2}:= s_{2} s_{2}^{F_{1}} s_{2}' L_{2}'' t_{2}' t_{2}$. In addition, let $L_{i}:=L_{i}'$ for each $i\in [3,k]$ and $i\ne j$. The paths $\left\{L_{1},\ldots,L_{k}\right\}$ give the desired $Y$-linkage.

\subsection *{\bf \uppercase{Subcase} C. The vertex opposite to $s_{1}$ in $F_{1}$ coincides with $ t_{1}$.}  Then $ t_{1}$ has no neighbour in $\A_1$. In fact, $F_{1}$ is the only facet in $\St_{1}$ containing $ t_{1}$. 

Because the vertex $s_{1}$ is not in Configuration $d$F, $ t_{1}$ has a neighbour $ t_{1}^{F_{1}}$ in $F_{1}$ that is not in $ X$. Here we  reason as in the scenario in which $ s_{2}=s_{1}^{o}$ and $ s_{2}$ has a neighbour not in $ X$. 

First, using the $(k-1)$-linkedness of $\Lk_1$ (\cref{prop:link-cubical})  find  disjoint paths $ L_{i}:= s_{i}- t_{i}$  in $\Lk_1$ for all $i\in [2,k]$\dmark. It may happen that $t_{1}^{F_{1}}$ is in one of the paths $L_{i}$ for $i\in [2,k]$. Second, consider neighbours  $s_{1}'$ and $ t_{1}'$ in $\A_1$ of $s_{1}$ and $ t_{1}^{F_{1}}$, respectively. 

 If $ t_{1}^{F_{1}}$ doesn't belong to any path $ L_{i}$, then  find a path $ L_{1}:=s_{1}- t_{1}$ that contains the edge $ t_{1} t_{1}^{F_{1}}$ and a subpath $L_{1}'$ in $\A_1$  between $s_{1}'$ and $ t_{1}'$; that is,  $L_{1}=s_{1}s_{1}'L_{1}'t_{1}'t_{1}^{F_{1}}t_{1}$. The desired $Y$-linkage is given by  $\{ L_{1}, \ldots, L_{k}\}$.

 If $ t_{1}^{F_{1}}$ belongs to one of the paths $ L_{i}$ with $i\in [2,k]$, say $ L_{j}$, then disregard this path $L_{j}$ and consider in $\A_1$ a neighbour $ s_{j}'$ of $ s_{j}$ and a neighbour $ t_{j}'$ of $ t_{j}$. From \cref{lem:projections-star}, it follows that the vertices $s_{1}'$, $ t_{1}'$, $ s_{j}'$ and $ t_{j}'$ can be taken pairwise distinct. Using the 2-linkedness of $\A_{1}$ for $d\ge 7$, find a path $L_{1}'$ in $\A_{1}$ between $s_{1}'$ and $ t_{1}'$ and a path $ L_{j}'$ in $\A_1$ between $ s_{j}'$ and $ t_{j}'$\dmark. Let $L_{1}:=s_{1}s_{1}'L_{1}'t_{1}'t_{1}^{F_{1}}t_{1}$ and let $L_{j}:=s_{j}s_{j}'L_{j}'t_{j}'t_{j}$ be the new $s_{j}-t_{j}$ path.  The paths  $\{ L_{1}, \ldots, L_{k}\}$ form the desired $Y$-linkage.

And finally, the proof  of \cref{lem:star-cubical} is complete.
\end{proof}
	
\section{Strong linkedness of cubical polytopes}

\begin{proof}[Proof of \cref{thm:cubical-strong-linkedness} (Strong linkedness of cubical polytopes)] Let $P$ be a cubical $d$-polytope. For odd $d$ \cref{thm:cubical-strong-linkedness} is a consequence of \cref{thm:cubical}. The result for $d=4$ is given by \cite[Theorem 16]{BuiPinUgo20a}. So assume $d=2k\ge 6$. Let $X$ be a set of $d+1$ vertices in $P$. Arbitrarily pair $2k$ vertices in $X$ to obtain $Y:=\{\{s_{1},t_{1}\},\ldots,\{s_{k},t_{k}\}\}$. Let $x$ be the vertex of $X$ not paired in $Y$. We find a $Y$-linkage $\{L_{1},\ldots, L_{k}\}$ where each path $L_{i}$ joins the pair $\{s_{i},t_{i}\}$ and avoids the vertex $x$.  

Using the $d$-connectivity of $G(P)$ and Menger's theorem, bring the $d=2k$ terminals in $X\setminus \{x\}$ to the link of $x$ in the boundary complex of $P$ through $2k$ disjoint paths $L_{s_{i}}$ and $L_{t_{i}}$ for $i\in [1,k]$. Let $s_{i}':=V(L_{s_{i}})\cap \lk(x)$  and $t_{i}':=V(L_{t_{i}})\cap \lk(x)$ for $i\in [1,k]$. Thanks to \cref{thm:cubical}, when $d\geq 6$, the link of $x$ is $k$-linked. Using the $k$-linkedness of $\lk(x)$, find disjoint paths $L_{i}':=s_{i}'-t_{i}'$ in $\lk(x)$. Observe that all these $k$ paths $\{L_{1}',\ldots,L_{k}'\}$ avoid $x$. Extend each path $L_{i}'$ with $L_{s_{i}}$ and $L_{t_{i}}$ to form a path $L_{i}:=s_{i}-t_{i}$ for each $i\in [1,k]$. The paths $\{L_{1},\ldots,L_{k}\}$ form the desired $Y$-linkage.
\end{proof}	

\section{Conflict of interest}
 On behalf of all authors, the corresponding author states that there is no conflict of interest.  


\providecommand{\bysame}{\leavevmode\hbox to3em{\hrulefill}\thinspace}
\providecommand{\MR}{\relax\ifhmode\unskip\space\fi MR }
\providecommand{\MRhref}[2]{%
  \href{http://www.ams.org/mathscinet-getitem?mr=#1}{#2}
}
\providecommand{\href}[2]{#2}

\appendix
\section{Proof of Lemma~\ref{lem:star-cubical} for the case $d=5$}\label{app:lemmad5}

The proof of the lemma for the case $d=5$ follows a similar structure as the case $d\geq 7$, but requires some technical adjustments. We rely on the following lemmas:

\begin{lemma}[{\cite[Lemma 14]{BuiPinUgo20a}}] Let P be a cubical $d$-polytope with $d\ge 4$. Let $X$ be a set of $d+1$  vertices in $P$, all contained in a facet $F$. Let $k:=\floor{(d+1)/2}$. Arbitrarily label and pair $2k$ vertices in $X$ to obtain $Y:=\{\{s_{1},t_{1}\},\ldots,\{s_{k},t_{k}\}\}$. Then, for at least $k-1$ of these pairs $\{s_{i},t_{i}\}$, there is an  $X$-valid  $s_{i}-t_{i}$ path in $F$.   
\label{lem:short-distance} 
\end{lemma}

\begin{proposition}[{\cite[Prop. 4 and Cor. 5 ]{BuiPinUgo20a}}]\label{prop:3-polytopes} Let $G$ be the graph of a 3-polytope and let $X$ be a set of four vertices of $G$. The set $X$ is linked in $G$ if and only if there is no facet of the polytope containing all the vertices of $X$. In particular, no nonsimplicial 3-polytope is 2-linked.
\end{proposition}

Given sets $A,B,X$ of vertices in a graph $G$,  the set  $X$  {\it separates} $A$ from $B$ if every $A-B$ path in the graph contains a vertex from $X$. A set $X$ separates two vertices $a,b$ not in $X$ if it separates $\{a\}$ from $\{b\}$. We call the set $X$ a {\it separator} of the graph. A set of vertices in a graph is {\it independent} if no two of its elements are adjacent.

 \begin{corollary}[{\cite[Corollary 10]{BuiPinUgo20a}}]\label{cor:separator-independent}  A separator of cardinality $d$ in a $d$-cube is an independent set.
 \end{corollary}

\begin{proof}[Proof of Lemma~\ref{lem:star-cubical} for $d=5$]

We proceed as in the proof for $d\geq 7$, and consider the same four cases. We let $k:=3$ and let $s_{1}$ be a vertex in a cubical $5$-polytope $P$ such that $s_{1}$ is not in Configuration $5$F. Recall that $\St_{1}$ denotes the star of $s_{1}$ in $\B(P)$. Let $X$ be any set of $6$ vertices in the graph $G(\St_{1})$ of  $\St_{1}$. The vertices in $X$ are our terminals. Also let $Y:=\{\{s_{1},t_{1}\},\{s_{2},t_{2}\},\{s_{3},t_{3}\}\}$ be a labelling and pairing of the vertices of $X$. We aim to find a $Y$-linkage $\{L_{1},L_{2},L_{3}\}$ in $G$ where $L_{i}$ joins the pair $\{s_{i},t_{i}\}$ for $i\in \{1,2,3\}$. Recall that a path is $X$-valid if it contains no inner vertex from $X$. 

We consider a facet $F_{1}$ of $\St_1$ containing $t_{1}$ and having the largest possible number of terminals. The four cases we consider in the Proof for the case $d\ge 7$ are: 
\begin{enumerate}

\item[\cref{case:new-linkedness-thm-1}.] $|X\cap V(F_{1})|=5$.
\item[\cref{case:new-linkedness-thm-2}.] $3\le |X\cap V(F_{1})|\le 4$.
\item[\cref{case:new-linkedness-thm-3}.] $|X\cap V(F_{1})|=2$ .
\item[\cref{case:new-linkedness-thm-4}.] $|X\cap V(F_{1})|=6$.  
 \end{enumerate} 
\cref{case:new-linkedness-thm-3} does not require any modification: all the arguments apply for $d\geq 5$. Let us consider the other three cases.

\begin{case}\label{case:new-linkedness-thm-1-d5}  $|X\cap V(F_{1})|= 5$.\end{case}
  
Without loss of generality,  assume that $  t_{2}\not \in V(F_{1})$.

In this case we proceed as for the case $d\ge 7$ until the final part of the proof where we find disjoint paths  $ L_{i}:=\pi_{R^{o}}^{F_{1}}( s_{i})-\pi_{R^{o}}^{F_{1}}( t_{i})$ ($i\in [1,k]$ and $i\ne 2$) in $R^{o}$ linking the $d-1$ vertices in $ X_{R^{o}}$. When $d=5$ we can only do that when the terminals in $R^o$ are not in cyclic order (in which case we proceed as in the proof for $d\ge 7$). Thus assume that the terminals are in cyclic order.
This in turn implies that $\pi_{R}^{F_{1}}( s_{3})\not\in\{ s_{2}, s_{2}'\}$ and $\pi_{R}^{F_{1}}( t_{3})\not\in\{ s_{2}, s_{2}'\}$, since $\dist_{F_{1}}(s_{1}, s_{2})=4$.  

Find a path $ L_{3}'$ in $R$  between $\pi_{R}^{F_{1}}( s_{3})$ and $\pi_{R}^{F_{1}}( t_{3})$ such that $ L_{3}'$ is disjoint from both $ s_{2}$ and $ s_{2}'$ and disjoint from $ t_{1}$ if $ t_{1}\in R$; here  use \cref{cor:separator-independent}, which ensures that the vertices $ s_{2}$, $ s_{2}'$ and $ t_{1}$, if they are all in $R$, cannot separate $\pi_{R}^{F_{1}}( s_{3})$ from $\pi_{R}^{F_{1}}( t_{3})$ in $R$,  since a separator of size three in $R$ must be an independent set. Extend the path $ L_{3}'$ in $R$ to a path $ L_{3}:= s_{3}\pi_{R}^{F_{1}}(s_{3})L_{3}'\pi_{R}^{F_{1}}(t_{3})t_{3}$ in $F_{1}$, if necessary. Find a path $ L_{1}':=s_{1}-\pi_{R^{o}}^{F_{1}}( t_{1})$ in $R^{o}$ disjoint from $\pi_{R^{o}}^{F_{1}}(s_{3})$ and $\pi_{R^{o}}^{F_{1}}(t_{3})$, using the 3-connectivity of $R^{o}$. Extend $L_{1}'$ to a path $L_{1}:=s_{1}L_{1}'\pi_{R^{o}}^{F_{1}}(t_{1})t_{1}$ in $F_{1}$, if necessary. The linkage $\{ L_{1}, L_{2},  L_{3}\}$  is a $Y$-linkage.  This completes the proof of \cref{case:new-linkedness-thm-1}. 

\begin{case}\label{case:new-linkedness-thm-2-d5}  $3\le | X\cap V(F_{1})|\le 4$.\end{case}

In this case we proceed as in the proof for $d\geq 7$, but some comments for $d=5$ are in order. By virtue of \cref{prop:3-polytopes}, we need to make sure that the sequence $\bar s_{2},\bar s_{3},\bar t_{2},\bar t_{3}$ in $ X^{+}_{R^{o}}$ is not in a 2-face of $R^{o}$ in cyclic order.  To ensure this, we need to be a bit more careful when selecting the vertices in $\bar Z$. Indeed, if there are already two vertices in $ X_{R^{o}}$ at distance three in $R^{o}$, no care is needed when selecting $\bar Z$, so proceed as in the case of $d\ge 7$. Otherwise, pick a vertex $\bar z \in \bar Z\subseteq V(R^{o})\setminus ( X_{R^{o}}\cup \{s_{1}^{o}\})$ such that $\bar z$ is the unique vertex in $R^{o}$ with $\dist_{R^{o}}(\bar z,x)=3$ for some vertex $x\in  X_{R^{o}}$; this vertex $x$ exists because $| X\cap V(F_{1})|\ge 3$. Selecting such a $\bar z\ne s_{1}^{o}$ is always possible because $s_{1}^{o}$ is not at distance three in $R^{o}$ from {\it any} vertex in $ X_{R^{o}}$: the unique vertex in $R^{o}$ at distance three from $s_{1}^{o}$ is $\pi_{R^{o}}^{F_{1}}(s_{1})$, and $\pi_{R^{o}}^{F_{1}}(s_{1})\not\in  X$ because the pair $\{R,R^{o}\}$ is not associated with $ X\cap V(F_{1})$. Once $\bar z$ is selected, the set $Z$ will contain a neighbour $z$ of $\bar z$. In this way,  some path $ S_{i}$ or $ T_{j}$ bringing terminals $ s_{i}$ or $ t_{j}$ in $\A_{1}$  into $R^{o}$ through $Z$ would use the vertex $z$, thereby ensuring that $x$ and $\bar z$ would be both in  $X^{+}_{R^{o}}$. This will cause the   the sequence $\bar s_{2},\bar s_{3},\bar t_{2},\bar t_{3}$ not to be in a 2-face, and thus, not in cyclic order.

\item[\cref{case:new-linkedness-thm-4}.] $|X\cap V(F_{1})|=6$.  

  The difficulty with $d=5$ stems from the $3$-faces of the polytope not being 2-linked (\cref{prop:3-polytopes}). Recall that in this case, all the terminals are in the facet $F_1$. The proof is divided into subcases depending on the nature of the vertex opposite to $s_1$ in $F_1$. Either it is not in $X$ (subcase A), or it is a terminal but not $t_1$ (subcase B), or it is $t_1$ (subcase C).

\subsection *{\bf \uppercase{Subcases A and B}. The vertex $s_{1}^{o}$ opposite to $s_{1}$ in $F_{1}$ either does not belong to $ X$ or belongs to $ X$ but is different from $ t_{1}$.}

  Let $X:=\{s_{1},s_{2},s_{3}, t_{1},t_{2},t_{3}\}$ be any set of six vertices in the graph $G$ of a cubical $5$-polytope $P$. Also let $Y:=\{\{s_{1},t_{1}\},\{s_{2},t_{2}\},\{s_{3},t_{3}\}\}$. We aim to find a $Y$-linkage $\{L_{1},L_{2},L_{3}\}$ in $G$ where $L_{i}$ joins the pair $\{s_{i},t_{i}\}$ for $i=1,2,3$. 
    
In both subcases there is a 3-face $R$ of $F_{1}$ containing both $s_{1}$ and $ t_{1}$. Let $J_{1}$ be the other facet in $\St_{1}$ containing $R$. Denote by $R_{J}$ and $R_{F}$ the $3$-faces in $J_{1}$ and $F_{1}$, respectively, that are disjoint from $R$. Then $s_{1}^{o}\in R_{F}$. We need the following claim.

\begin{claim}\label{cl:d=5} If a 3-cube contains three pairs of terminals, there must exist two pairs of terminals in the 3-cube, say $\{s_1,  t_{1}\}$ and $\{ s_{2}, t_{2}\}$,  that are not  arranged in the cyclic order $s_1, s_{2}, t_{1}, t_{2}$ in a 2-face of the cube.
\end{claim}

\begin{remark}\label{rmk:cubical-common-neighbours} If $x$ and $y$ are vertices of a cube, then they share at most two neighbours. In other words, the complete bipartite graph $K_{2,3}$ is not a subgraph of the cube; in fact, it is not an induced subgraph of any simple polytope \cite[Cor.~1.12(iii)]{PfePilSan12}. 
\end{remark}

\begin{claimproof} If no terminal in the cube is in Configuration 3F, we are done. So  suppose that one is, say $s_{1}$, and that the sequence $s_1,x_{1},  t_{1},x_{2}$ of vertices of $X$ is present in cyclic order in a 2-face. Without loss of generality, assume that $ s_{2}\not\in\{x_{1},x_{2}\}$. Then $ s_{2}$ cannot be adjacent to both $ s_{1}$ and $ t_{1}$, since the bipartite graph $K_{2,3}$ is not a subgraph of $G(Q_{3})$ (\cref{rmk:cubical-common-neighbours}). Thus  the sequence $s_1, s_{2}, t_{1}, t_{2}$ cannot be in a 2-face in cyclic order. 
\end{claimproof}  
 
{\bf Suppose all the six terminals are in the 3-face $R$.} By virtue of \cref{cl:d=5},  we may assume that the pairs $\{s_{1},  t_{1}\}$ and $\{  s_{{2}}, t_{{2}}\}$ are not arranged in the cyclic order $s_1, s_{2}, t_{1}, t_{2}$ in a 2-face of $R$.  \cref{prop:3-polytopes}  ensures that the pairs $\{\pi_{R_{J}}^{J_{1}}( s_{1}),\pi_{R_{J}}^{J_{1}}( t_{1})\}$ and $\{\pi_{R_{J}}^{J_{1}}( s_{2}),\pi_{R_{J}}^{J_{1}}( t_{2})\}$ in $R_{J}$ can be linked in $R_{J}$ through disjoint paths $ L_{1}'$ and $ L_{2}'$, since the sequence $\pi_{R_{J}}^{J_{1}}( s_{1}),\pi_{R_{J}}^{J_{1}}( s_{2}),\pi_{R_{J}}^{J_{1}}(t_{1}),\pi_{R_{J}}^{J_{1}}( t_{2})$ cannot be in a 2-face of $R_{J}$ in cyclic order. Moreover, by the connectivity of $R_{F}$, there is a path $ L_{3}'$ in $R_{F}$ linking the  pair $\{\pi_{R_{F}}^{F_{1}}( s_{3}),\pi_{R_{F}}^{F_{1}}( t_{3})\}$. The linkage $\{ L_{1}',  L_{2}',  L_{3}'\}$ can naturally be extended to a $ Y$-linkage $\{ L_{1},  L_{2},  L_{3}\}$  as follows.

\[L_{i}:=\begin{cases}
	s_{i}\pi_{R_{J}}^{J_{1}}( s_{i})L_{i}'\pi_{R_{J}}^{J_{1}}(t_{i})t_{i}, &\text{for $i=1,2$};\\
	s_{3}\pi_{R_{F}}^{F_{1}}( s_{3})L_{3}'\pi_{R_{F}}^{F_{1}}(t_{3})t_{3}, &\text{otherwise}.
\end{cases}
\]

{\bf Suppose that $R$ contains a pair $\{ s_{i}, t_{i}\}$ for $i=2,3$, say $\{ s_{2},  t_{2}\}$}. There are at most five terminals in $R$, and consequently, applying  \cref{lem:short-distance} to the polytope $F_{1}$ and its facet $R$,  we obtain an $ X$-valid  path $ L_{1}:=s_{1}- t_{1}$  in $R$ or an $ X$-valid path  $ L_{2}:= s_{2}- t_{2}$  in $R$. For the sake of concreteness, say an $ X$-valid path $ L_{2}$ exists in $R$. From the connectivity of $R_{F}$ and $R_{J}$ follows the existence of a path $ L_{3}'$ in $R_{F}$ between $\pi_{R_{F}}^{F_{1}}( s_{3})$ and $\pi_{R_{F}}^{F_{1}}( t_{3})$, and of a path $ L_{1}'$ in $R_{J}$ between $\pi_{R_{J}}^{J_{1}}(s_{1})$ and $\pi_{R_{J}}^{J_{1}}( t_{1})$ (recall that $t_1\in R\subset J_1$). The linkage $\{ L_{1}', L_{2}', L_{3}'\}$ can be extended to a linkage $\{s_{1}- t_{1}, s_{2}- t_{2}, s_{3}- t_{3}\}$ in $\St_{1}$.
	
{\bf Suppose that the ridge $R$ contains no other pair from $ Y$ and that the ridge $R_{F}$ contains a  pair $( s_{i}, t_{i})$ ($i=2,3$)}.  Without loss of generality, assume $ s_{2}$ and $ t_{2}$ are in $R_{F}$. 

First suppose that $s_{3}\in R$, which implies that $ t_{3}\in R_{F}$. Further suppose that there is a path $T_{3}$ of length at most two from $t_{3}$ to $R$ that is disjoint from $ X\setminus\{ s_{3}, t_{3}\}$. Let $ \{ t_{3}'\}:=V(T_{3})\cap V(R)$. Use the 2-linkedness of the 4-polytope $J_{1}$ \cite[Prop. 6]{BuiPinUgo20a} to find disjoint paths $ L_{1}:=s_{1}- t_{1}$ and $ L_{3}':= s_{3}- t_{3}'$ in $J_{1}$. Let $ L_{3}:= s_{3} L_{3}' t_{3}'T_{3} t_{3}$. Use the 3-connectivity of $R_{F}$ to find an $ X$-valid path $ L_{2}:= s_{2}- t_{2}$ in $R_{F}$ that is disjoint from $V(T_{3})$; note that $|V(T_{3})\cap V(R_{F})|\le 2$. The paths $\{ L_{1}, L_{2}, L_{3}\}$ give the desired  $Y$-linkage. Now suppose there is no such path $T_{3}$ from $ t_{3}$ to $R$. Then, the projection $\pi_{R}^{F_{1}}( t_{3})$ is in  $\{s_{1}, t_{1}\}$, say $\pi_{R}^{F_{1}}( t_{3})=t_{1}$; the projection $\pi_{R_{F}}^{F_{1}}(s_{1})$ is a neighbour of $ t_{3}$ in $R_{F}$; and both $ s_{2}$ and $ t_{2}$ are neighbours of $ t_{3}$ in $R_{F}$. This configuration implies that $s_{1}$ and $ t_{1}$ are adjacent in $R$. Let $ L_{1}:=s_{1} t_{1}$. Find a path $ L_{2}:= s_{2}- t_{2}$ in $R_{F}$ that is disjoint from $t_{3}$, using the 3-connectivity of $R_{F}$.   
Then using \cref{lem:projections-star} find a neighbour $ s_{3}'$ in $\A_{1}$ of $ s_{3}$ and a neighbour $ t_{3}'$ in $\A_{1}$ of $t_{3}$; note that, since $\dist_{F_{1}}(s_{1}, t_{3})\le 2$, we have that $ t_{3}\ne s_{1}^{o}$, and since $\{s_1,s_3\}\in V(R)$, $s_3\ne s_{1}^o$. Find a path $ L_{3}$ in $\St_{1}$ between $ s_{3}$ and $ t_{3}$  that contains a subpath $L_{3}'$ in $\A_{1}$ between $ s_{3}'$ and $ t_{3}'$; here use the connectivity of $\A_{1}$ (\cref{prop:star-minus-facet}): $L_{3}:=s_{3}s_{3}'L_{3}'t_{3}'t_{3}$. The linkage $\{ L_{1}, L_{2}, L_{3}\}$ is the desired  $Y$-linkage.

Assume that $ s_{3}\in R_{F}$; by symmetry we can further assume that $ t_{3}\in R_{F}$. The connectivity of $R$ ensures the existence of a path $ L_{1}:= s_{1}- t_{1}$ therein.  In the case of $s_{1}^{o}\in  X$, without loss of generality, assume $s_{1}^{o}= s_{2}$. The 3-connectivity of $R_{F}$ ensures the existence of an $ X$-valid path $L_{2}:= s_{2}- t_{2}$ therein. Use \cref{lem:projections-star} to find pairwise distinct neighbours $ s_{3}'$ of $ s_{3}$ and $ t_{3}'$ of $ t_{3}$ in $\A_{1}$; these exist since $ s_{3}\ne s_{1}^{o}$ and $ t_{3}\ne s_{1}^{o}$. Using the connectivity of $\A_{1}$ (\cref{prop:star-minus-facet}), find a path $ L_{3}:= s_{3}- t_{3}$ in $\St_{1}$ that contains a subpath $ s_{3}'- t_{3}'$ in $\A_{1}$. The linkage $\{ L_{1}, L_{2}, L_{3}\}$ is the desired  $Y$-linkage.

{\bf Assume neither $R$ nor $R_{F}$ contains a pair $\{ s_{i}, t_{i}\}$ ($i=2,3$)}. Without loss of generality, assume  that $ s_{2}, s_{3}\in R$,  that $ t_{2}, t_{3}\in R_{F}$ and that $ t_{2}\ne s_{1}^{o}$.  

First suppose  that there exists a path $S_{3}$ in $F_{1}$ from $ s_{3}$ to $R_{F}$ that is  of length at most two and is disjoint from $ X\setminus \{ s_{3}, t_{3}\}$. Let $\{\hat s_{3}\}:=V(S_{3})\cap V(R_{F})$.   Find pairwise distinct neighbours $ s_{2}'$ and $ t_{2}'$ of $ s_{2}$ and $ t_{2}$, respectively, in $\A_{1}$. And find a path $ L_{2}:= s_{2}- t_{2}$ in $\St_{1}$ that contains a subpath $ s_{2}'- t_{2}'$ in $\A_{1}$ (using the connectivity of $\A_{1}$). 
Using the 3-connectivity of $R_{F}$ link the pair $\{\hat s_{3}, t_{3}\}$ in $R_{F}$ through a path $ L_{3}'$ that is disjoint from $t_{2}$. Let $L_{3}:=s_{3}S_{3}\hat s_{3}L_{3}'t_{3}$. Since \cref{cor:separator-independent}  ensures that any separator of size three in a 3-cube must be independent, we can find a path $ L_{1}:=s_{1}- t_{1}$ in $R$ that is disjoint from $s_{2}$ and $V(S_{3})\cap V(R)$; the set $V(S_{3})\cap V(R)$ has either cardinality one or contains an edge. The paths $\{ L_{1}, L_{2}, L_{3}\}$ form the desired  $Y$-linkage.

Assume that there is no such path $S_{3}$. In this case, the neighbours of $ s_{3}$ in $F_{1}$ are $s_{1}, t_{1}, s_{2}$ from $R$ and $ t_{2}$ from $R_{F}$. Use \cref{lem:projections-star} to find a neighbour  $ s_{3}'$ of $ s_{3}$ in $\A_{1}$. Again use \cref{lem:projections-star}  either to find  a neighbour $ t_{3}'$ of $ t_{3}$ if $ t_{3}\ne s_{1}^{o}$ or to find a neighbour $ t_{3}'$ of a neighbour $u$ of $ t_{3}$ in $R_{F}$ (with $u\ne t_{2}$) if $  t_{3}=s_{1}^{o}$.  Let $T_{3}$ be the path of length at most two from $ t_{3}$ to $\A_{1}$ defined as $T_{3}= t_{3} t_{3}'$ if $ t_{3}\ne s_{1}^{o}$ and $T_{3}= t_{3}u t_{3}'$ if $ t_{3}= s_{1}^{o}$. Find  a path $ L_{3}$ in $\St_{1}$ between $ s_{3}$ and $ t_{3}$  that contains a subpath in $\A_{1}$ between $ s_{3}'$ and $ t_{3}'$; here use the connectivity of $\A_{1}$ (\cref{prop:star-minus-facet}). We next find a path $S_{2}$ in $F_{1}$ from $ s_{2}$ to $R_{F}$ that is  of length at most two and is disjoint from $V(T_{3})\cup \{s_{1}, t_{1}, s_{3}\}$. There are exactly four disjoint  $ s_{2}-R_{F}$ paths of length at most two, one through each of the neighbours of $ s_{2}$ in $F_{1}$. One such path is $ s_{2} s_{3} t_{2}$. Among the remaining three $ s_{2}-R_{F}$ paths, since none of them contains $s_{1}$ or $ t_{1}$ and since $|V(T_{3})\cap V(R_{F})|\le 2$, we find the path $S_{2}$.  Let $\hat s_{2}:=V(S_{2})\cap V(R_{F})$. Find a path $ L_{2}':=\hat s_{2}- t_{2}$ in $R_{F}$ that is disjoint from $V(T_{3})$, using the 3-connectivity of $R_{F}$. Let $ L_{2}:= s_{2}S_{2}\hat s_{2} L_{2}' t_{2}$. Since the vertices in $(V(S_{2})\cap V(R))\cup\{ s_{3}\}$ cannot separate $s_{1}$ from $ t_{1}$ in $R$ (\cref{cor:separator-independent}), find a path $ L_{1}:=s_{1}- t_{1}$ in $R$ disjoint from $V(S_{2})\cap V(R)\cup\{ s_{3}\}$; the set $V(S_{2})$ has cardinality one or contains one edge. The paths $\{ L_{1}, L_{2}, L_{3}\}$ form the desired  $Y$-linkage.

\subsection *{\bf \uppercase{Subcase} C. The vertex opposite to $s_{1}$ in $F_{1}$ coincides with $ t_{1}$}  
 
Since $s_1$ is not in configuration $d$3 we may suppose  that $ t_{1}$ has a neighbour $ t_{1}'$ not in $ X$. We reason as in Subcases A and B. We give the details for the sake of completeness.  

Let $R$ denote the $3$-face in $F_{1}$ containing both $s_{1}$ and $ t_{1}'$; $\dist_{R}(s_{1}, t_{1}')=3$. Let $R_{F}$ be the $3$-face of $F_{1}$ disjoint from $R$.  Let $J_{1}$ be the other facet in $\St_{1}$ containing $R$ and let $R_{J}$ be the $3$-face of $J_{1}$ disjoint from $R$.
  
{\bf Suppose $R$ contains a pair $\{ s_{i}, t_{i}\}$ ($i=2,3$), say $( s_{2}, t_{2})$.} There are at most five terminals in $R$ (as $t_1$ is in $R_F$). Since the smallest face in $R$ containing $s_{1}$ and $ t_{1}'$ is 3-dimensional,  the sequence  $\pi_{R_{J}}^{J_{1}}(s_{1}), \pi_{R_{J}}^{J_{1}}( s_{2}), \pi_{R_{J}}^{J_{1}}( t_{1}'),\pi_{R_{J}}^{J_{1}}( t_{2})$ cannot appear in a 2-face of $R_{J}$ in cyclic order. As a consequence, the pairs  $\{\pi_{R_{J}}^{J_{1}}(s_{1}), \pi_{R_{J}}^{J_{1}}( t_{1}')\}$ and $\{\pi_{R_{J}}^{J_{1}}( s_{2}), \pi_{R_{J}}^{J_{1}}( t_{2})\}$ can be linked in $R_{J}$ through disjoint paths $ L_{1}'$ and $ L_{2}'$, thanks to \cref{prop:3-polytopes}. Let $ L_{1}:=s_{1}\pi_{R_{J}}^{J_{1}}(s_{1}) L_{1}'\pi_{R_{J}}^{J_{1}}( t_{1}') t_{1}' t_{1}$ and $ L_{2}:= s_{2}\pi_{R_{J}}^{J_{1}}(s_{2}) L_{2}'\pi_{R_{J}}^{J_{1}}( t_{2}) t_{2}$. From the 3-connectivity of $R_{F}$ follows the existence of a path $ L_{3}'$ in $R_{F}$ between $\pi_{R_{F}}^{F_{1}}( s_{3})$ and $\pi_{R_{F}}^{F_{1}}( t_{3})$ that avoids $ t_{1}$. Let $ L_{3}:= s_{3}\pi_{R_{F}}^{F_{1}}( s_{3}) L_{3}'\pi_{R_{F}}^{F_{1}}( t_{3}) t_{3}$. The paths $\{ L_{1}, L_{2}, L_{3}\}$ form the desired   $ Y$-linkage.

{\bf Suppose that the ridge $R$ contains no pair  $\{ s_{i}, t_{i}\}$ ($i=2,3$) and that the ridge $R_{F}$ contains a  pair $\{ s_{i}, t_{i}\}$ ($i=2,3$), say $\{ s_{2}, t_{2}\}$}. Then, there are at most five terminals in $R_{F}$. If there are at most four terminals in $R_{F}$, the 3-connectivity of $R_{F}$ ensures the existence of an $ X$-valid path $ L_{2}:= s_{2}- t_{2}$ in $R_{F}$; if there are exactly five terminals in $R_{F}$, applying \cref{lem:short-distance} to the polytope $F_{1}$ and its facet $R_{F}$ gives either an $ X$-valid path $ L_{2}:= s_{2}- t_{2}$ or an $ X$-valid path  $ L_{3}:= s_{3}- t_{3}$ in $R_{F}$. As a result, regardless of the number of terminals in $R_{F}$, we can assume there is an $ X$-valid path $ L_{2}:= s_{2}-t_{2}$ in $R_{F}$. Find pairwise distinct neighbours $ s_{3}'$ and $ t_{3}'$  in $\A_{1}$ of $ s_{3}$ and $ t_{3}$, respectively, and a path $ L_{3}$ in $\St_{1}$ between $ s_{3}$ and $ t_{3}$  that contains a subpath in $\A_{1}$ between $ s_{3}'$ and $ t_{3}'$; here use the connectivity of $\A_{1}$ (\cref{prop:star-minus-facet}). In addition, let $ L_{1}'$ be a path in $R$ between  $s_{1}$ and $ t_{1}'$;  here use the 3-connectivity of $R$ to avoid any terminal in $R$. Let $L_{1}:=s_{1}L_{1}'t_{1}'t_{1}$.  The  $Y$-linkage is given by the paths $\{ L_{1}, L_{2}, L_{3}\}$.

{\bf Assume neither $R$ nor $R_{F_{1}}$ contains a pair $\{ s_{i}, t_{i}\}$ ($i=2,3$).}  Without loss of generality,  we can assume $ s_{2}, s_{3}\in R$ and $ t_{2}, t_{3}\in R_{F}$. 

There exists a path $S_3$ from $s_{3}$ to $R_F$ that is of length at most two and is disjoint from $\{s_1,t_1,t'_1,s_2,t_2\}$. If $\pi_{R_F}(s_3) \ne t_2$, then $S_3 = s_3\pi_{R_F}(s_3)$. Otherwise, there are exactly three disjoint paths of length 2 from $s_3$ to $R_F$. At most two of them contain a vertex in $N_{R}(s_3)\cap (X\cup \{t'_1\})$ (since $\dist(s_1,t_1)=3$, they cannot be both neighbours of $s_3$). Thus we can take $S_3$ as the path $s_3u\pi_{R_F}(u)$ through a neighbour $u$ of $s_3$ in $R$ such that $u\notin X\cup \{t'_1\}$ and $\pi_{R_F}(u) \notin \{t_1,t_2\} = \{\pi_{R_F}(s_3),\pi_{R_F}(t'_1)\}$.

 Let $\{\hat s_{3}\}:=V(S_{3})\cap V(R_{F})$. Find an $ X$-valid path $ L_{3}':=\hat s_{3}- t_{3}$ in $R_{F}$ using its 3-connectivity. Let $ L_{3}:= s_{3}S_{3}\hat s_{3} L_{3}' t_{3}$.  Then find neighbours $ s_{2}'$ and $ t_{2}'$ of $ s_{2}$ and $ t_{2}$, respectively, in $\A_{1}$, and a path $ L_{2}:= s_{2}- t_{2}$ in $\St_{1}$ that contains a subpath $ s_{2}'- t_{2}'$ in $\A_{1}$ (using the connectivity of $\A_{1}$). Since \cref{cor:separator-independent}  ensures that any separator of size three in a 3-cube must be independent, we can find an $ L_{1}':=s_{1}- t_{1}'$ in $R$ that is disjoint from $ s_{2}$ and $V(S_{3})\cap V(R)$; the set $V(S_{3})\cap V(R)$ has either cardinality one or contains an edge. Let  $ L_{1}:=s_{1} L_{1}' t_{1}' t_{1}$. The paths $\{ L_{1}, L_{2}, L_{3}\}$ form the desired   $ Y$-linkage.

This concludes the proof of \cref{lem:star-cubical} for $d=5$.
\end{proof}

\end{document}